\documentclass[10pt,leqno]{amsart}
\usepackage{graphicx}
\usepackage{indentfirst,csquotes,amsthm, amscd, amsfonts, amssymb}

\usepackage{tikz,tikz-cd}
\usepackage{geometry}   
\usepackage{pb-xy}
\usepackage{mathtools}
\usepackage[shortlabels]{enumitem}
\usepackage{float}
\usepackage{hyperref}                      
\usepackage{dsfont}
\renewcommand\arraystretch{1.4}
 \newcommand{\tbigcap}{\mathop{\textstyle \bigcap}}
 \newcommand{\tbigvee}{\mathop{\textstyle \bigvee}}
 \newcommand{\tbigwedge}{\mathop{\textstyle \bigwedge}}

\newcommand{\tsum}{\mathop{\textstyle \sum}}

\newcommand{\set}[1]{\{\,#1\,\}}
\newcommand{\bigset}[1]{\bigl\{\,#1\,\bigr\}}

\newcommand{\alg}[1]{\ensuremath{\mathfrak{#1}}} 
\newcommand{\op}{\ensuremath{ {o\!p}}} 
\newcommand{\cat}[1]{\ensuremath{\mathsf{#1}}} 

\theoremstyle{plain}
          \newtheorem{theorem}{Theorem}[section]
          \newtheorem{lemma}[theorem]{Lemma}
          \newtheorem{proposition}[theorem]{Proposition}
	\newtheorem{corollary}[theorem]{Corollary}
	
	\newtheorem{fact}[theorem]{Fact}
	\newtheorem*{fact*}{Fact}

        \theoremstyle{definition}
          \newtheorem{example}[theorem]{Example}
                   
          \newtheorem*{addition*}{Addition}
\newtheorem*{standingassumption*}{Standing Assumption}

        \theoremstyle{remark}
\newtheorem*{remark*}{Remark}          
\newtheorem{remark}[theorem]{Remark}
\newtheorem{remarks}[theorem]{Remarks}

\newtheorem{comment*}[theorem]{Comment}
\numberwithin{equation}{section}
\newcounter{num}[section]

\newcommand{\donotbreakdash}[1]{#1\nobreakdash-\hspace{0pt}}
\tikzset{every picture/.style={line width=0.7pt}}

\makeatletter
\newbox\xrat@below
\newbox\xrat@above
\renewcommand{\xrightarrow}[2][]{%
 \setbox\xrat@below=\hbox{\ensuremath{\scriptstyle #1}}%
 \setbox\xrat@above=\hbox{\ensuremath{\scriptstyle #2}}%
\pgfmathsetlengthmacro{\xrat@len}{max(\wd\xrat@below,\wd\xrat@above)+.6em}%
 \mathrel{\tikz [->,baseline=-.75ex]
     \draw (0,0) -- node[below=-2pt] {\box\xrat@below}
        node[above=-2pt] {\box\xrat@above}
      (\xrat@len,0) ;}}
\renewcommand{\xleftarrow}[2][]{%
 \setbox\xrat@below=\hbox{\ensuremath{\scriptstyle #1}}%
 \setbox\xrat@above=\hbox{\ensuremath{\scriptstyle #2}}%
\pgfmathsetlengthmacro{\xrat@len}{max(\wd\xrat@below,\wd\xrat@above)+.6em}%
 \mathrel{\tikz [<-,baseline=-.75ex]
     \draw (0,0) -- node[below=-2pt] {\box\xrat@below}
        node[above=-2pt] {\box\xrat@above}
      (\xrat@len,0) ;}}
\newcommand{\xrightarrowtail}[2][]{%
 \setbox\xrat@below=\hbox{\ensuremath{\scriptstyle #1}}%
 \setbox\xrat@above=\hbox{\ensuremath{\scriptstyle #2}}%
\pgfmathsetlengthmacro{\xrat@len}{max(\wd\xrat@below,\wd\xrat@above)+.6em}%
 \mathrel{\tikz [>->,baseline=-.75ex]
     \draw (0,0) -- node[below=-2pt] {\box\xrat@below}
        node[above=-2pt] {\box\xrat@above}
      (\xrat@len,0) ;}}
\newcommand{\xlongmapsto}[2][]{%
 \setbox\xrat@below=\hbox{\ensuremath{\scriptstyle #1}}%
 \setbox\xrat@above=\hbox{\ensuremath{\scriptstyle #2}}%
\pgfmathsetlengthmacro{\xrat@len}{max(\wd\xrat@below,\wd\xrat@above)+.6em}%
 \mathrel{\tikz [|->,baseline=-.75ex]
     \draw (0,0) -- node[below=-2pt] {\box\xrat@below}
        node[above=-2pt] {\box\xrat@above}
      (\xrat@len,0) ;}}

\newcommand{\xleftrightarrows}[2][]{%
 \setbox\xrat@below=\hbox{\ensuremath{\scriptstyle #1}}%
 \setbox\xrat@above=\hbox{\ensuremath{\scriptstyle #2}}%
\pgfmathsetlengthmacro{\xrat@len}{max(\wd\xrat@below,\wd\xrat@above)+.6em}%
 \mathrel{\begin{tikzpicture} [baseline=-.75ex]
     \draw[->] (0,0.05) -- node[above=-2pt] {\box\xrat@above}
      (\xrat@len,0.05) ;
     \draw[<-] (0,-0.05) -- node[below=-2pt] {\box\xrat@below}
      (\xrat@len,-0.05) ;
      \end{tikzpicture}
      }}
\newcommand{\xleftleftarrows}[2][]{%
 \setbox\xrat@below=\hbox{\ensuremath{\scriptstyle #1}}%
 \setbox\xrat@above=\hbox{\ensuremath{\scriptstyle #2}}%
\pgfmathsetlengthmacro{\xrat@len}{max(\wd\xrat@below,\wd\xrat@above)+.6em}%
 \mathrel{\begin{tikzpicture} [baseline=-.75ex]
     \draw[->] (0,0.05) -- node[above=-2pt] {\box\xrat@above}
      (\xrat@len,0.05) ;
     \draw[->] (0,-0.05) -- node[below=-2pt] {\box\xrat@below}
      (\xrat@len,-0.05) ;
      \end{tikzpicture}
      }}

\begin{document}
\title[Involutive quantales and quantale-enriched involutive topological spaces]{Involutive quantales and quantale-enriched\\ involutive topological spaces} 
\author[J. Guti\'errez Garc{\'\i}a]{Javier Guti\'errez Garc{\'\i}a}
\thanks{The first named author acknowledges support from the Basque Government (grant IT1483-22). }
\date{\today}
\address{Departamento de Matem\'aticas, Universidad del Pa\'{\i}s Vasco (UPV/EHU), 48080, Bilbao, SPAIN}
\email{javier.gutierrezgarcia@ehu.eus}
\author[U. H\"ohle]{Ulrich H\"ohle}
\date{\today}
\address{Fakult\"{a}t f\"{u}r Mathematik und Naturwissenchaften, Bergische Universit\"{a}t, D-42097, Wuppertal, GERMANY}
\email{uhoehle@uni-wuppertal.de}
\maketitle

\let\thefootnote\relax

\begin{abstract}
In this paper, we provide a comprehensive analysis of involutive quantales, with a particular focus on quantic frames. We extend the axiomatic foundations of quantale-enriched topological spaces to include closure under the anti-homomorphic involution, facilitating a balanced topologization of the spectrum of unital \donotbreakdash{$C^*$}algebras that encompasses both closed right and left ideals through the concept of quantic frames. Specifically, certain subspaces of pure states are identified as strongly Hausdorff separated, involutive quantale-enriched  topological spaces.
\end{abstract} 
\bigskip

\section*{Introduction}

Quantales are semigroups in the sym\-me\-tric, monoidal, closed category of complete lattices and join-preserving maps. Involutive quantales are quantales with an involution that is both anti-homomorphic and join-preserving. It is well known that not all quantales admit such involutions. 
A typical counterexample is a non-commutative quantale defined on a complete chain. However, every quantale possessing an idempotent element distinct from the universal lower bound (e.g.\ a balanced quantale) can be embedded into an involutive quantale (see\ Remark in Sec.~\ref{sec:2}). This observation highlights the significance of unitalization, through which any quantale can be realized as a subquantale of some involutive quantale.
Therefore, the assumption of an anti-homomorphic and join-preserving involution should not be viewed as a restrictive condition. Rather, it reflects that the quantale in question has an appropriate structural ``size''.

Further, every strongly spatial quantale can be characterized 
by a special type of quantale-enriched topological space --- a so-called quantized topological space 
(cf.\ Proposition~\ref{newproposition6}). 
Since the spectra of non-commutative and unital \donotbreakdash{$C^*$}algebras are expressed by involutive quantales (cf.\ \cite[Def.\ 2.2]{Mulvey01} and \cite[p.\ 157]{EGHK}), this situation motivates the extension of the axiomatic foundations of quantale-enriched topological spaces.

This extension ensures that quantale-enriched topologies are also closed under the anti-homomorphic and join-preserving involution of the underlying unital and involutive quantale.
This approach leads to a change in the concept of subbase and in all those notions depending on subbases ---  e.g.\ weak regularity (cf.\ \cite[Def.\ 6.5]{Arrieta}). 
On the other hand, we gain a natural topologization of the spectrum of a unital \donotbreakdash{$C^*$}algebra $A$, which includes not only all closed right ideals of $A$ 
(as in the approach considered by Mulvey and Pelletier (cf.\ \cite[Sec.~7]{Mulvey02})), but also all closed left ideals of $A$. 
This balanced treatment of closed left and closed right ideals in the topologization process is 
achieved through the application of the important concept of quantic frames (cf.\ \cite[Def.\ 2.5.13]{EGHK}) ---
 a concept which reduces to traditional frames provided the quantale multiplication is commutative. Therefore, from the perspective of applications,  this axiomatic extension of quantale-enriched topologies seems to be reasonable and leads consequently to the concept of involutive quantale-enriched topological spaces.

The paper is organized as follows. We begin with some preliminaries on quantales. In Section~\ref{sec:2}, we present a fundamental theorem on the construction of involutive quantales, originating from \cite[Thm.~2.3.33]{EGHK}. As an application, we introduce the concept of a \emph{tensorially involutive} quantale associated with an involutive and semi-unital quantale. Furthermore, in the case of quantale epimorphisms with an involutive quantale as its domain, we characterize the existence of an anti-homomorphic and join-preserving involution on its codomain. This result transforms the given epimorphism into an involutive quantale homomorphism. As an application we present a formulation of the spectrum of a unital and non-commutative \donotbreakdash{$C^*$}algebra, which is equivalent to the traditional spectrum of \donotbreakdash{$C^*$}algebras in the commutative case. It is interesting to see  that in the framework of the category of balanced and bisymmetric quantales with strong quantale homomorphisms this construction has a categorical base. In this sense, the spectrum of a unital \donotbreakdash{$C^*$}algebra is a \emph{quantic frame}. This means, among other things, that its spectrum is uniquely determined by its left-sided and right-sided elements in the form of a pushout square.

Subsequently, we recall the axioms of a quantale-enriched topological space including the lower separation axioms, and introduce the concept of involutive quantale-enriched topological spaces. Since the quantization of $\boldsymbol{2}$ is an involutive quantale, involutive quantized topological spaces form an important subclass of this new type of structure. It is remarkable to see that the strong Hausdorff separation axiom has a convergent-theoretical characterization in the case of quantized topologies. Moreover, every involutive quantale has a topological representation by an involutive quantized topological space (cf.\ Proposition~\ref{newproposition7}). An application of this construction to  quantic frames, and especially to the spectrum of non-commutative \donotbreakdash{$C^*$}algebras, is given. In the case of the \donotbreakdash{$C^*$}algebra $B(\mathcal{H})$ of all bounded, linear operators on a Hilbert space $\mathcal{H}$, a strongly Hausdorff separated, involutive quantized topological space is presented, whose underlying set depends on the Hilbert space dimension of $\mathcal{H}$.
 
\section{Preliminaries on quantales}
\label{sec:1}

Let \cat{Sup} be the category of complete lattices and join-preserving maps. We recall that $\cat{Sup}$ is a sym\-me\-tric, monoidal closed category. The \donotbreakdash{$2$}chain $\boldsymbol{2}=\set{0,1}$ is the unit object in $\cat{Sup}$. Furthermore, if $X$ and $Y$ are objects of $\cat{Sup}$, then the \donotbreakdash{$XY$}component $X\otimes Y\xrightarrow{\, c_{XY}\,} Y\otimes X$ of the symmetry in $\cat{Sup}$ is uniquely determined on elementary
 tensors of $X\otimes Y$ by $c_{XY}(x\otimes y)=y\otimes x$ for all $x\in X$ and $y\in Y$ (for 
more details see \cite[Sect.~2.1]{EGHK}). The universal upper (resp.\ lower) bound in a complete lattice is always denoted by $\top$ (resp.\ $\bot$).

Further, we recall that a \emph{quantale} is a semigroup in $\cat{Sup}$ (cf.\ \cite[chap.~VII]{MacLane}) --- i.e. 
 a complete lattice provided with an binary and associative operation $\alg{Q}\otimes \alg{Q} \xrightarrow{\, m\,} \alg{Q}$. 
Due to the universal property of the tensor product in $\cat{Sup}$ the operation $m$ can always be identified with a semigroup operation $\alg{Q}\times \alg{Q}\xrightarrow{\, \ast\,} \alg{Q}$ in $\cat{Set}$ which is join-preserving in each variable separately. In the following,
$\ast$ will be called the \emph{quantale multiplication} of $\alg{Q}$, and so the \emph{left-} and \emph{right-implication} corresponding to $\ast$ are given by:
\[\alpha\swarrow\beta=\tbigvee\set{\gamma\in \alg{Q}\mid\gamma\ast \beta\le\alpha}  \quad \text{and}\quad\alpha\searrow\beta=\tbigvee\set{\gamma \in \alg{Q}\mid \alpha\ast \gamma\le \beta},\qquad \alpha,\beta\in \alg{Q}.\] 
An element $\delta$ of a quantale is \emph{dualizing} if $\delta$ satisfies  the following property for all $\alpha\in \alg{Q}$:
\[\delta\swarrow (\alpha\searrow \delta)=\alpha=(\delta\swarrow\alpha)\searrow \delta.\]
A quantale $(\alg{Q},m)$ is \emph{balanced} if its universal upper bound is idempotent --- i.e. $\top\ast \top=\top$.

Let $(\alg{Q},m)$ and $(\alg{R},n)$ be quantales. 
A join-preserving map $\alg{Q}\xrightarrow{\,h\,} \alg{R}$ is called a \emph{quantale homomorphism} if the following diagram is commutative:
\[\begin{tikzcd}[column sep=25pt,row sep=14 pt
]
\alg{Q}\otimes \alg{Q}\arrow[swap]{d}{m}\arrow{r}{h\otimes h}&\alg{R}\otimes \alg{R}\arrow{d}{n}\\
\alg{Q}\arrow{r}{h}&\alg{R}
\end{tikzcd}\]
i.e. if  $h(\alpha\ast\beta)=h(\alpha)\ast h(\beta)$ for each $\alpha,\beta\in\alg{Q}$.
Consequently quantales and quantale homomorphisms form a category denoted by $\cat{Quant}$. A quantale homomorphism $h$ is \emph{strong} if $h$ preserves the respective universal upper bounds --- i.e. $h(\top)=\top$. 

A quantale $(\alg{R},n)$ is a \emph{subquantale} of  $(\alg{Q},m)$ if $\alg{R}$ is a subset of $\alg{Q}$ and the inclusion map $\alg{R}\xhookrightarrow{\,\,\,} \alg{Q}$ is a monomorphism in $\cat{Quant}$.

A join-preserving map $\alg{Q}\xrightarrow{\,h\,} \alg{R}$ is called a \emph{quantale anti-homomorphism} if the following diagram is commutative:
\[\begin{tikzcd}[column sep=25pt,row sep=14 pt
]
\alg{Q}\otimes \alg{Q}\arrow[swap]{d}{m}\arrow{r}{h\otimes h}&\alg{R}\otimes \alg{R}\arrow{r}{c_{\alg{R}\alg{R}}}&\alg{R}\otimes \alg{R}\arrow{d}{n}\\
\alg{Q}\arrow{rr}{h}&&\alg{R}
\end{tikzcd}\]
\noindent 
where $c_{\alg{R}\alg{R}}$ is the \donotbreakdash{$\alg{R}\alg{R}$}component of the symmetry in $\cat{Sup}$ --- i.e.\  if  $h(\alpha\ast\beta)=h(\beta)\ast h(\alpha)$ for each $\alpha,\beta\in\alg{Q}$.

A quantale $(\alg{Q},m)$ is called \emph{involutive} if there exists a quantale anti-homomorphism $\alg{Q}\xrightarrow{\, \ell\,} \alg{Q}$ satisfying $\ell\circ\ell=1_{\alg{Q}}$. Then we write $(\alg{Q},m,\ell)$ and call  $\ell$ also an \emph{involution} on $\alg{Q}$. A  join-preserving map $\alg{Q}\xrightarrow{\,h\,} \alg{R}$ between involutive quantales is involutive if the following diagram is commutative:
\[\begin{tikzcd}[column sep=25pt,row sep=14 pt
]
\alg{Q}\arrow[swap]{d}{\ell_{\alg{Q}}}\arrow{r}{h}&\alg{R}\arrow{d}{\ell_{\alg{R}}}\\
\alg{Q}\arrow{r}{h}&\alg{R}
\end{tikzcd}\]
\noindent  i.e. if $h(\alpha^{\prime})=h(\alpha)^{\prime}$ for each $\alpha\in\alg{Q}$. It is not difficult to show that a join-preserving map $\alg{Q}\xrightarrow{\,h\,} \alg{R}$  is involutive if and only if the right adjoint map $h^{\vdash}$ of $h$ is involutive (cf.\ \cite[Lem.~2.2.14]{EGHK}).
 
An element $\alpha$ of an involutive quantale is called \emph{hermitian} if $\alpha=\alpha^{\prime}$ holds. A subquantale $\alg{S}$ of an involutive quantale $\alg{Q}$ is called \emph{involutive} if $\alpha^{\prime}\in \alg{S}$ for all $\alpha\in \alg{S}$.

Let $(\alg{Q},m)$ be a quantale. 
The \emph{opposite quantale} of $(\alg{Q},m)$ is the quantale $(\alg{Q},m^{op})$ where $m^{op}$ is determined by the following commutative diagram:
\[\begin{tikzcd}[column sep=25pt,row sep=14 pt
]
\alg{Q}\otimes \alg{Q}\arrow[swap]{rd}{m^{\op}}\arrow{r}{c_{\alg{Q}\alg{Q}}}&\alg{Q}\otimes \alg{Q}\arrow{d}{m}\\
&\alg{Q}
\end{tikzcd}\]
In the following we will call $m^{op}$ the \emph{opposite binary operation} of $m$ and denote 
the corresponding quantale multiplication
 by $\ast^{op}$ --- i.e.\  $\alpha\mathop{\ast^{op}}\beta=\beta\ast \alpha$ for each $\alpha,\beta\in\alg{Q}$.
Obviously, $m^{op}$ coincides with $m$ if and only if $(\alg{Q},m)$ is commutative. 

A \emph{unital} quantale $(\alg{Q},m)$ is a monoid in $\cat{Sup}$. We identify the unit $\mathds{1}
\xrightarrow{\,e\,} \alg{Q}$ with the element $e(1)\in \alg{Q}$, which we will also denote by $e$. Therefore, we write $(\alg{Q},m,e)$ for a unital quantale. The tensor product of unital quantales is again unital.

Let $\alg{Q}=(\alg{Q},m)$ be a quantale, and let $\ast$ denote the corresponding quantale multiplication. Then $\alg{Q}$ is said to have the following properties:
\begin{enumerate}[label=\textup{--},leftmargin=20pt,labelwidth=10pt,itemindent=0pt,labelsep=5pt,topsep=5pt,itemsep=3pt
]
\item \emph{left-sided} if every element $\alpha\in \alg{Q}$ is left-sided --- i.e. $\top\ast \alpha\le \alpha$.
\item \emph{right-sided} if every element $\alpha\in \alg{Q}$ is right-sided --- i.e. $  \alpha\ast \top\le \alpha$.
\item \emph{two-sided} if $\alg{Q}$ is left-sided \emph{and}  right-sided.
\item \emph{pre-idempotent} if   every element $\alpha\in \alg{Q}$ is pre-idempotent --- i.e. $\alpha\le \alpha\ast \alpha$.
\item \emph{idempotent} if  every element $\alpha\in \alg{Q}$ is idempotent --- i.e. $\alpha= \alpha\ast \alpha$.
\item \emph{semi-unital} if $\alpha\le \top\ast\alpha$ and $\alpha\le \alpha\ast \top$ for all $\alpha\in \alg{Q}$.
\item \emph{semi-integral} if $\alpha\ast \top\ast\beta\le \alpha\ast \beta$ for all $\alpha,\beta\in \alg{Q}$.
\item \emph{bisym\-me\-tric} if $\alpha\ast \beta_1\ast\beta_2\ast \gamma=\alpha\ast\beta_2\ast \beta_1\ast \gamma$ for all $\alpha,\beta_1,\beta_2,\gamma\in \alg{Q}$.
\end{enumerate}
Every unital (resp.\ pre-idempotent) quantale is semi-unital, every semi-unital quantale is ba\-lanced, and every left-sided (resp.\ right-sided) quantale is semi-integral. The subquantale of $\alg{Q}$ of all left-sided (resp.\ right-sided, two-sided) elements of $\alg{Q}$ is denoted by $\mathds{L}(\alg{Q})$ (resp.\ $\mathds{R}(\alg{Q})$, $\mathds{I}(\alg{Q})$).
A quantale $\alg{Q}$ is a \emph{factor} if $\mathds{I}(\alg{Q})$ contains only the universal bounds of $\alg{Q}$ --- i.e. $\mathds{I}(\alg{Q})=\set{\bot,\top}$. 
\begin{lemma} \label{newlemma1} Let $\alg{Q}$ be a semi-unital and semi-integral  quantale. If $\alg{Q}$ is  a factor, then  the following property holds for all $\alpha,\beta\in \alg{Q}$\textup:
\[ \alpha \ast \gamma \ast \beta= \alpha \ast \beta, \qquad \gamma\in \alg{Q}\setminus\{\bot\}.\]
\end{lemma}

\begin{proof} Let us consider $\gamma\in \alg{Q}\setminus\{\bot\}$. Then $\top\ast \gamma \ast \top\in \mathds{I}(\alg{Q})$, and since $\alg{Q}$ is a factor and semi-unital,  we obtain $\top\ast \gamma \ast \top=\top$. Finally, since $\alg{Q}$ is semi-integral and semi-unital, we conclude that
\[\alpha \ast \gamma \ast \beta\le \alpha \ast \top \ast \beta\le \alpha \ast \beta\le \alpha \ast \top\ast \beta=\alpha \ast \top\ast \gamma \ast \top\ast \beta\le \alpha \ast \gamma \ast \beta.\qedhere\]
\end{proof}

Further, $\alg{Q}$ \emph{does not have zero divisors if,} for $\alpha\neq \bot$ and $\beta\neq \bot$, the relation $\alpha \ast \beta\neq \bot$ follows.
\begin{lemma}\label{newlemma2} Let $\alg{Q}$ be a semi-unital quantale.
\begin{enumerate}[label=\textup{(\alph*)},leftmargin=20pt,labelwidth=10pt,itemindent=0pt,labelsep=5pt, topsep=5pt,itemsep=3pt
]
\item \label{(a)} If $\mathds{I}(\alg{Q})$ does not have zero divisors, then $\mathds{L}(\alg{Q})$ \textup(resp.\ $\mathds{R}(\alg{Q})$\textup) also does not have zero divisors.
\item \label{(b)} If $\mathds{L}(\alg{Q})$ does not have zero divisors, then for $\alpha\in \alg{Q}\setminus\{\bot\} $ and for $\beta\in\mathds{L}(\alg{Q})\setminus \{\bot\}$ the relation $\alpha \ast \beta\neq \bot$ holds.
\item \label{(c)} 
 If $\mathds{R}(\alg{Q})$ does not have zero divisors, then for $\alpha\in \mathds{R}(\alg{Q)}\setminus \{\bot\}$ and for $\beta\in \alg{Q}\setminus\{\bot\}$ the relation $\beta \ast \alpha\neq \bot$ holds.
\end{enumerate}
\end{lemma}

\begin{proof} 
Let $\alpha,\beta\in \mathds{L}(\alg{Q})$ such that $\alpha\ast \beta=\bot$. Then  $\alpha\ast \beta\ast \top=\bot$, 
and the left-sidedness of $\beta$ implies $(\alpha \ast \top)\ast (\beta\ast \top)=\bot$. 
Since $\alpha \ast \top$ and $\beta\ast \top$ are two-sided and $\mathds{I}(\alg{Q})$ does not 
have zero divisors, we obtain $\alpha\ast \top=\bot$ or $ \beta\ast \top=\bot$, 
and so $\alpha=\bot$ or $\beta=\bot$ follows --- i.e.  $\mathds{L}(\alg{Q})$ does not have zero divisors. 
The case $\mathds{R}(\alg{Q})$ can be treated analogously.
The verification of (b) and (c) is left to the reader.
\end{proof}
  
A prominent example of a left-sided quantale which does not have zero divisors, is given in Example~\ref{newexample1} infra.

An element $p$ of a quantale $\alg{Q}$ is called \emph{prime} if $p\neq \top$ and for all $\alpha,\beta\in \alg{Q}$ the following implication holds (cf.\ \cite[p.~100]{EGHK}):
\[\text{If }\alpha\ast\beta\le p,\text{ then }\alpha\ast\top\le p\text{ or }\top\ast \beta\le p.\]
A prime element of a quantale $\alg{Q}$ is \emph{strongly prime} if for all $\alpha,\beta\in \alg{Q}$ the stronger implication holds (cf.\ \cite[p.~150]{EGHK}):
\[\text{If }\alpha\ast\beta\le p,\text{ then } \alpha \vee (\alpha\ast \top)\le p \text{ or } \beta \vee (\top\ast \beta)\le p.\] 

The set $\sigma(\alg{Q})$ (resp. $\sigma_s(\alg{Q})$) of all prime  (resp.\ strongly prime) 
elements is called the \emph{spectrum} (resp. \emph{strong spectrum}) of $\alg{Q}$. 
The \emph{hermitian spectrum} of an involutive quantale $(\alg{Q},m,\ell)$ is the set $\sigma_h(\alg{Q})$ of all hermitian and strongly prime elements.

A quantale $\alg{Q}$ is called \emph{strongly spatial} (resp.\ \emph{spatial}), if every element of $\alg{Q}$ is a meet of strongly prime (resp.\ prime) elements of $\alg{Q}$.
\begin{remarks}\label{newremarks1} (1) If $\alg{Q}$ is semi-unital, then every prime element is strongly prime --- i.e.\  $\sigma_s(\alg{Q})= \sigma(\alg{Q})$. Hence any spatial and semi-unital quantale is strongly spatial.\\[1mm]
(2) Let $\widehat{\alg{Q}}$ be the semi-unitalization of $\alg{Q}$ (cf.\ \cite[pp.~98, 150]{EGHK}) --- i.e. $\alg{Q}$ is extended by an idempotent element $\widehat{\top}$, which serves as the universal upper bound in $\widehat{\alg{Q}}=\alg{Q}\cup\{\widehat{\top}\}$ and interacts with $\alg{Q}$ as follows:
\[ \alpha \ast \widehat{\top}:=\alpha \vee (\alpha \ast \top)\quad\text{ and }\quad \widehat{\top}\ast \alpha:= \alpha \vee (\top \ast \alpha), \qquad \alpha \in \alg{Q}.\]
Obviously $\widehat{\alg{Q}}$ is semi-unital, and the universal upper bound $\top$ in $\alg{Q}$ is prime and two-sided in $\widehat{\alg{Q}}$.
Moreover, $\sigma\bigl(\widehat{\alg{Q}}\bigr)=\sigma_s\bigl(\widehat{\alg{Q}}\bigr)=\sigma_s(\alg{Q})\cup\{\top\}$, hence $\alg{Q}$ is strongly spatial if and only if $\widehat{\alg{Q}}$ is  spatial.\\[1mm]
(3) Let $L$ be a complete lattice and let $[L,L]$ be the quantale of all join-preserving self-maps $L\xrightarrow{\,\,\,} L$ ordered pointwise and provided with the composition as quantale multiplication. Then $[L,L]$ is semi-unital and a factor. If $L$ contains at least \emph{three} elements, then $[L,L]$ is \emph{not semi-integral} and the spectrum of $\alg{Q}$ is \emph{empty}.
\end{remarks}

Concerning the \emph{existence} of prime elements, the \emph{semi-integrality} of quantales plays a significant role (see e.g.,\ Remarks~\ref{newremarks1}\,(3)). We quote the following fact (cf. \cite[Thm.~2.3.21]{EGHK}.
\begin{fact} \label{newfact1} Every maximal left-sided \textup(resp.\ right-sided\textup) element of a semi-integral quantale is prime.
\end{fact}
 
Recall that every spatial quantale is semi-integral and bisym\-me\-tric (cf. \cite[Lem.~1.3\,(i)]{GH24}).
The next proposition summarizes some  properties of strongly spatial quantales.
\begin{proposition}\label{newproposition1} Let $\alg{Q}$ be strongly spatial. Then\textup:
\begin{enumerate}[label=\textup{(\roman*)},leftmargin=20pt,labelwidth=10pt,itemindent=0pt,labelsep=5pt,topsep=5pt,itemsep=3pt]
\item  $\alg{Q}$ is semi-unital.
\item $\alg{Q}$ is pre-idempotent.
\item If $\alg{Q}$ is either left-sided or right-sided, then $\alg{Q}$ is idempotent.
\end{enumerate}
\end{proposition}

\begin{proof}
To verify (i), we fix $\alpha\in \alg{Q}$ and observe that if $p$ is strongly prime in $\alg{Q}$ such that $\alpha\ast\top\le p$ then $\alpha\le p$. Consequently, 
 $\alpha\le\tbigwedge\set{p\in \sigma_s(\alg{Q})\mid \alpha\ast\top\le p}=\alpha\ast\top$.
 Analogously, we can verify $\alpha\le\top  \ast \alpha$. 
The remaining properties follow from (i) and  \cite[Lem.~1.3\,(ii) and (iii)]{GH24}.
\end{proof}

With regard to Proposition~\ref{newproposition1}\,(i) the following equivalence holds.

\begin{corollary} A quantale $\alg{Q}$ is strongly spatial if and only if it is spatial and semi-unital.
\end{corollary}

\begin{comment*} There exist spatial quantales, which are \emph{not} semi-unital. 
The simplest example is the trivial quantale on the \donotbreakdash{$3$}chain $C_3=\set{\bot,a,\top}$ with $\bot < a < \top$ --- i.e.\  the quantale multiplication is determined by $\top\ast\top=\bot$. Since $\bot$ and $a$ are prime elements, it follows that it is a spatial quantale. However,  $\bot$ is not strongly prime, hence the quantale fails to be strongly spatial.
\end{comment*}

Before we proceed we present an example illustrating the previous results.
\begin{example}\label{newexample1}  Let $\mathcal{H}$ be an infinite dimensional Hilbert space and $B(\mathcal{H})$ be the unital \donotbreakdash{$C^*$}algebra of all bounded linear operators on $\mathcal{H}$. Further, let $\mathds{L}(B(\mathcal{H}))$ be the complete lattice of all closed left ideals of $B(\mathcal{H})$ ordered by set-inclusion and provided  with the usual ideal multiplication:
\[I_1\ast I_2= \text{top.\ closure}\bigl(\bigset{\tsum\limits_{i=1}^n f_i\cdot g_i\mid  n\in\mathds{N},\,f_i\in I_1,\, g_i\in I_2}\bigr),\qquad I_1,I_2\in \mathds{L}(B(\mathcal{H})).\]
Then $\mathds{L}(B(\mathcal{H}))=(\mathds{L}(B(\mathcal{H})),\ast)$ is a left-sided quantale. With regard to {Fact}~\ref{newfact1}, every \emph{maximal} left ideal is \emph{prime}. Since every closed left ideal of $B(\mathcal{H})$ is an intersection of maximal left ideals (cf.\ \cite[Thm.~10.2.10\,(iii)]{Kadison2}), the quantale $\mathds{L}(B(\mathcal{H}))$ is  spatial. Now, using the fact that $B(\mathcal{H})$ is unital, we conclude that $\mathds{L}(B(\mathcal{H}))$ is \emph{semi-unital}. Hence $\mathds{L}(B(\mathcal{H}))$ is strongly spatial, and consequently, Proposition~\ref{newproposition1}\,(iii) implies that the ideal multiplication is idempotent.

We now show that the ideal multiplication does not have zero divisors. With regard  to Lemma~\ref{newlemma2}\,(a), it is sufficient to verify that the subquantale $\mathds{I}(B(\mathcal{H}))$ of all closed two-sided ideals of $B(\mathcal{H})$ does not have zero divisors. Since every non-zero closed two-sided ideal contains the closed two-sided ideal of all compact operators acting on $\mathcal{H}$ (cf.\ \cite[p.~747]{Kadison2}), the property of not having zero divisors follows immediately from the idempotency of the ideal multiplication. Hence, the zero ideal is a (strongly) prime element of $\mathds{L}(B(\mathcal{H}))$. Analogously, it can be shown that the quantale $\mathds{R}(B(\mathcal{H}))$  of all closed right ideals does not have zero divisors, and so the zero ideal is also a (strongly) prime element of $\mathds{R}(B(\mathcal{H}))$.
 
Further, it is well known that every unit vector $x\in \mathcal{H}$ induces a pure state of $B(\mathcal{H})$ by (cf.\ \cite[Exercise~4.6.68]{Kadison1}):
\[B(\mathcal{H})\xrightarrow{\,\omega_x\,} \mathds{C},\quad \omega_x(f)=\langle f(x),x\rangle, \qquad f\in B(\mathcal{H}),\]
where $\langle\,\,\,,\,\,\,\rangle$ is the inner product of $\mathcal{H}$. Since the one-to-one relationship between pure states $\varrho$ and maximal left ideals $I$ of $B(\mathcal{H})$ is given by means of the construction of left kernels (cf.\ \cite[Thm.~10.2.10\,(i) and (ii)]{Kadison2}):
\[I_{\varrho}=\set{f\in B(\mathcal{H})\mid \varrho(f^*\cdot f)=0},\]
every vector state $\omega_x$ determines a maximal left ideal $I_x$ by
\[I_x=\set{f\in B(\mathcal{H})\mid \langle (f^*\cdot f)(x),x\rangle=0}=\set{f\in B(\mathcal{H})\mid f(x)=\vec{0}}.\]
\end{example}

We end this section by recalling that the tensor product of  quantales $(\alg{Q},m)$ and $(\alg{R},n)$ is given by the tensor product $\alg{Q}\otimes \alg{R}$ of the underlying lattices, provided with an appropriate binary and associative operation 
\[(\alg{Q}\otimes \alg{R})\otimes  (\alg{Q}\otimes \alg{R}) \xrightarrow{\,m\otimes n,} \alg{Q}\otimes \alg{R}.\]
The  quantale multiplication $\star$ corresponding to $m\otimes n$ is uniquely determined on elementary tensors of $\alg{Q}\otimes \alg{R}$ as follows (cf.\ \cite[p.~92]{EGHK}):
\[(\alpha_1\otimes \beta_1)\star (\alpha_2\otimes \beta_2)= (\alpha_1\ast \alpha_2)\otimes (\beta_1\ast \beta_2),\qquad \alpha_1,\alpha_2\in \alg{Q},\, \beta_1,\beta_2\in \alg{R}.\]
The \emph{canonical} embeddings $\alg{Q}\xrightarrow{\, j_{\alg{Q}}\,} \alg{Q}\otimes \alg{R}$ and $\alg{R}\xrightarrow{\,j_{\alg{R}}\,} \alg{Q}\otimes \alg{R}$, defined  by
\[j_{\alg{Q}}(\alpha)= \alpha\otimes \top\quad \text{and}\quad j_{\alg{R}}(\beta)=\top\otimes \beta,\qquad \alpha\in \alg{Q},\,\beta\in \alg{R},\]
 are strong quantale homomorphisms if and only if both quantales $(\alg{Q},m)$ and $(\alg{R},n)$ are balanced.

The following permanence properties of the tensor product of quantales hold. If both quantales are pre-idempotent (resp.\ semi-unital, semi-integral, bisym\-me\-tric, and being a factor),
 then 
their tensor product again is pre-idempotent (resp.\ semi-unital, semi-integral, bisym\-me\-tric and a factor). 
In particular, the tensor product of a left-sided quantale with a right-sided quantale is always semi-integral. 

Finally, we recall a characterization of prime elements in the tensor product of quantales (cf.\ \cite[Lem.~2.3.22 and Thm.~2.3.25]{EGHK}).
\begin{theorem} \label{newtheorem1} Let $(\alg{Q},m)$ and $(\alg{R},n)$ be semi-unital quantales such that $(\alg{Q},m)$ is left-sided and $(\alg{R},n)$ is right-sided. Then an element $p\in \alg{Q}\otimes \alg{R}$ is a prime element of $(\alg{Q}\otimes \alg{R}, m\otimes n)$ if and only if there exist  prime elements $q$ in $(\alg{Q},m)$ and  $r$ in $(\alg{R},n)$ such that $p=(q\otimes \top)\vee (\top\otimes r)$. Moreover, $q$ and $r$ are uniquely determined by $p$.
\end{theorem}

\section{On the existence of involutive quantales}
\label{sec:2}

Since every commutative quantale is involutive, the existence of non-commutative but involutive quantales represents an interesting feature in the theory of quantales.
 
It is well known that with every quantale homomorphism $(\alg{Q},m)\xrightarrow{\, h\,} (\alg{R},n)$, we can associate a nucleus $c$ on $(\alg{Q},m_{\alg{}})$ by
\[c(\alpha)=h^{\vdash}(h(\alpha)),\qquad \alpha\in \alg{Q},\]
where $\alg{R}\xrightarrow{\,h^{\vdash}\,} \alg{Q}$ is the right adjoint map of $h$. A nucleus on an involutive quantale $(\alg{Q},m,\ell)$ is involutive if and only if $c(\alpha)^{\prime}\le c(\alpha^{\prime})$ for all $\alpha\in \alg{Q}$.
\begin{proposition} \label{newproposition2}Let $(\alg{Q},m,\ell)$ be an involutive quantale, and let $(\alg{Q},m)\xrightarrow{\, \pi\,} (\alg{R},n)$ be an epimorphism in the sense of $\cat{Quant}$ \textup(in particular, $\pi$ is surjective as map\textup). 

Then the nucleus associated with $\pi$ is involutive if and only if there exists an involution on $(\alg{R},n)$ such that $\pi$ is an involutive quantale homomorphism.
\end{proposition}

\begin{proof} Sufficiency is evident. To verify necessity, 
we proceed as follows. Let 
$c$ be the nucleus associated with $\pi$. Define the map $\alg{R} \xrightarrow{\,\ell_{\alg{R}}\,} \alg{R}$ by
$
\ell_{\alg{R}}:= \pi \circ\ell\circ \pi^{\vdash}.
$
This map is obviously isotone. Since $c$ is involutive, the  following relation holds:
\stepcounter{num}
\begin{equation} \label{n2.1CC}
\ell_{\alg{R}}\circ \pi=\pi \circ \ell \circ c=\pi \circ c \circ \ell =\pi\circ \ell.
\end{equation} 
Since $\ell$ is an involution and $\pi$ is surjective, we have $\ell_{\alg{R}}\circ \ell_{\alg{R}}=\pi\circ \ell \circ \ell \circ \pi^{\vdash}=\pi\circ \pi^{\vdash}=\text{id}_{\alg{R}}$, which shows that $\ell_{\alg{R}}$ is a join-preserving involution on $(\alg{R},n)$.
To prove that $\ell_{\alg{R}}$ is anti-homomorphic, we use  the anti-homomorphic property of $\ell$ and again the surjectivity of $\pi$. Let $\beta_1,\beta_2\in \alg{R}$. Hence there exist $\alpha_1,\alpha_2\in \alg{Q}$ such that $\pi(\alpha_i)=\beta_i$ $(i=1,2)$. Then:
  \begin{align*}\ell_{\alg{R}}(\beta_1\ast \beta_2)&=(\ell_{\alg{R}}\circ \pi)(\alpha_1\ast\alpha_2)=(\pi\circ\ell)(\alpha_1\ast\alpha_2)=\pi(\ell(\alpha_2))\ast {\pi(\ell(\alpha_1))=\ell_{\alg{R}}(\beta_2)\ast \ell_{\alg{R}}(\beta_1)}.
  \end{align*}
  Finally, (\ref{n2.1CC}) shows that  $\pi$ is an involutive quantale homomorphism w.r.t.\ $\ell_{\alg{R}}$.
\end{proof}

\begin{theorem}\label{newtheorem2} Let $(\alg{Q},m)$ and $(\alg{R},n)$ be balanced quantales, and let $\alg{Q}\otimes \alg{R}\xrightarrow {\,\ell\,} \alg{Q}\otimes \alg{R}$ 
be an involution on their tensor product --- i.e.\ $(\alg{Q}\otimes \alg{R},m\otimes n, \ell)$ is an involutive quantale.
 Further, let $\alg{Q}\xrightarrow{\,\vartheta_{\alg{Q}}\,} \alg{R}$ and $\alg{R}\xrightarrow{\,\vartheta_{\alg{R}}\,} \alg{Q}$ 
  be quantale anti-homomorphisms such that 
the following diagram is commutative\textup:
\stepcounter{num}
\begin{equation} \label{neweq2.1}
\begin{tikzcd}[column sep=25pt,row sep=14 pt
]
\alg{Q}\otimes \alg{R}\arrow{r}{\ell}&\alg{Q}\otimes \alg{R}\arrow{r}{\ell}&\alg{Q}\otimes \alg{R}\\
\alg{Q}\arrow{u}{j_{\alg{Q}}}\arrow{r}{\vartheta_{\alg{Q}}}&\alg{R}\arrow{u}{j_{\alg{R}}}\arrow{r}{\vartheta_{\alg{R}}}&\alg{Q}\arrow[swap]{u}{j_{\alg{Q}}}
\end{tikzcd}
\end{equation}
\noindent If both $(\alg{S},m_{\alg{S}})\xrightarrow{\, q_{\alg{Q}}\,} (\alg{Q},m)$ and $(\alg{S},m_{\alg{S}})\xrightarrow{\, q_{\alg{R}}\,} (\alg{R},n)$ 
are strong quantale 
 homomorphisms and $\alg{Q}\otimes \alg{R}\xrightarrow{\,\pi\,} V$ is the coequalizer of $\alg{S} \xleftleftarrows[\, j_{\alg{R}}\circ q_{\alg{R}}\,]{\,j_{\alg{Q}}\circ q_{\alg{Q}}\,} \alg{Q}\otimes \alg{R}$ 
in the sense of $\cat{Quant}$, then the nucleus associated with $\pi$ is 
 involutive if and only if the coequalizers of $\alg{S} \xleftleftarrows[\, j_{\alg{R}}\circ q_{\alg{R}}\,]{\,j_{\alg{Q}}\circ q_{\alg{Q}}\,} \alg{Q}\otimes \alg{R}$ and $\alg{S} \xleftleftarrows[\, j_{\alg{R}}\circ \vartheta_{\alg{Q}}\circ q_{\alg{Q}}\,]{\,j_{\alg{Q}}\circ \vartheta_{\alg{R}}\circ q_{\alg{R}}\,} \alg{Q}\otimes \alg{R}$ are isomorphic.
\end{theorem}

\begin{proof} Let $\alg{N}$ be the set of all nuclei on $(\alg{Q}\otimes \alg{R},m\otimes n)$, and let $\alg{Q}\otimes \alg{R}\xrightarrow{\, \overline{\pi}\,} W$ be the coequalizer of $\alg{S} \xleftleftarrows[\, j_{\alg{R}}\circ \vartheta_{\alg{Q}}\circ q_{\alg{Q}}\,]{\,j_{\alg{Q}}\circ \vartheta_{\alg{R}}\circ q_{\alg{R}}\,} \alg{Q}\otimes \alg{R}$. 
First, we recall the construction of coequalizers in $\cat{Quant}$ and explicitly define the respective nuclei $c_0$ and $\overline{c}_0$ associated with $\pi$ and $\overline{\pi}$, respectively (cf.\ \cite[Thm.~2.2.6]{EGHK}):
\begin{align*}
c_0 &= \tbigwedge\set{c\in \alg{N}\mid c\circ j_{\alg{Q}}\circ q_{\alg{Q}}=c\circ j_{\alg{R}}\circ q_{\alg{R}}},\\
\overline{c}_0 &=\tbigwedge\set{c\in \alg{N}\mid c\circ j_{\alg{Q}}\circ \vartheta_{\alg{R}}\circ q_{\alg{R}}=c\circ j_{\alg{R}}\circ\vartheta_{\alg{Q}}\circ q_{\alg{Q}}}
\end{align*}
where the meet is computed pointwise. In particular, the following properties hold: 
\[c_0\circ j_{\alg{Q}}\circ q_{\alg{Q}}=c_0\circ j_{\alg{R}}\circ q_{\alg{R}} \quad \text{and}\quad \overline{c}_0\circ j_{\alg{Q}}\circ \vartheta_{\alg{R}}\circ q_{\alg{R}}=\overline{c}_0\circ j_{\alg{R}}\circ\vartheta_{\alg{Q}}\circ q_{\alg{Q}}.\]

\noindent(a) {\bf Necessity}: If $c_0$ is  involutive, then we derive the following relation from the commutativity of  diagram (\ref{neweq2.1}):
\begin{align*}
c_0\circ j_{\alg{Q}}\circ \vartheta_{\alg{R}}\circ q_{\alg{R}}&=c_0\circ \ell\circ j_{\alg{R}}\circ q_{\alg{R}}=\ell\circ c_0\circ  j_{\alg{R}}\circ q_{\alg{R}}=\ell\circ c_0\circ  j_{\alg{Q}}\circ q_{\alg{Q}}\\
&=c_0\circ  \ell\circ j_{\alg{Q}}\circ q_{\alg{Q}}=c_0\circ j_{\alg{R}}\circ \vartheta_{\alg{Q}}\circ q_{\alg{Q}}.
\end{align*}
Hence, $\overline{c}_0\le c_0$ follows. On the other hand, since
$\ell\circ \overline{c}_0\circ \ell\in\alg{N}$ and
\begin{align*}
\ell\circ \overline{c}_0\circ \ell\circ j_{\alg{Q}} \circ q_{\alg{Q}}&=\ell\circ\overline{c}_0\circ j_{\alg{R}}\circ \vartheta_{\alg{Q}}\circ q_{\alg{Q}}=\ell\circ\overline{c}_0\circ j_{\alg{Q}}\circ \vartheta_{\alg{R}}\circ q_{\alg{R}}=\ell\circ\overline{c}_0\circ\ell\circ j_{\alg{R}}\circ q_{\alg{R}},
\end{align*}
  $c_0\le \ell\circ \overline{c}_0\circ \ell$ follows. Finally, using again the property that $c_0$ is involutive, we obtain 
\[c_0=\ell\circ\ell\circ c_0=\ell\circ c_0\circ \ell\le \ell\circ\ell\circ \overline{c}_0\circ \ell\circ \ell=\overline{c}_0.\]
Hence, $c_0=\overline{c}_0$ follows --- i.e.\ the regular quotient objects w.r.t.\ $\pi$ and $\overline{\pi}$ coincide, and so the coequalizers of $\alg{S} \xleftleftarrows[\, j_{\alg{R}}\circ q_{\alg{R}}\,]{\,j_{\alg{Q}}\circ q_{\alg{Q}}\,} \alg{Q}\otimes \alg{R}$ and $\alg{S} \xleftleftarrows[\, j_{\alg{R}}\circ \vartheta_{\alg{Q}}\circ q_{\alg{Q}}\,]{\,j_{\alg{Q}}\circ \vartheta_{\alg{R}}\circ q_{\alg{R}}\,} \alg{Q}\otimes \alg{R}$ are isomorphic.\\[1mm]
(b) {\bf Sufficiency}: Assume $c_0=\overline{c}_0$. Referring again to the commutativity of  diagram (\ref{neweq2.1}), we obtain 
\[\ell\circ c_0\circ \ell\circ j_{\alg{Q}}\circ q_{\alg{Q}}=\ell\circ c_0\circ j_{\alg{R}} \circ\vartheta_{\alg{Q}}\circ q_{\alg{Q}}=\ell\circ c_0\circ j_{\alg{Q}}\circ \vartheta_{\alg{R}}\circ q_{\alg{R}}=\ell\circ c_0\circ \ell\circ j_{\alg{R}}\circ q_{\alg{R}}.\]
Since $\ell\circ c_0\circ \ell$ is a nucleus, the relation $c_0\le \ell\circ c_0\circ \ell$ follows. Hence, $c_0$ is involutive.
\end{proof}

\begin{theorem} \label{newtheorem3}Let $(\alg{Q},m)$ and $(\alg{R},n)$ be balanced quantales, and let $\alg{Q}\xrightarrow{\,\vartheta_{\alg{Q}}\,} \alg{R}$ and $\alg{R}\xrightarrow{\,\vartheta_{\alg{R}}\,} \alg{Q}$ be quantale anti-homomorphisms satisfying the following conditions\textup:
\stepcounter{num}
\begin{equation} \label{neweq2.2} \vartheta_{\alg{R}}\circ \vartheta_{\alg{Q}}=1_{\alg{Q}}\quad \text{and} \quad \vartheta_{\alg{Q}}\circ \vartheta_{\alg{R}}=1_{\alg{R}}.
\end{equation}
If $\bigl(\alg{Q}\otimes \alg{R},m\otimes n)$ is the tensor product of $(\alg{Q},m)$ and $(\alg{R},n)$, then there exists a quantale anti-homomorphism $\alg{Q}\otimes \alg{R}\xrightarrow {\,\ell\,} \alg{Q}\otimes \alg{R}$ determined by\textup:
\stepcounter{num}
\begin{equation}\label{neweq2.3}
\ell=c_{\alg{R}\alg{Q}} \circ (\vartheta_{\alg{Q}}\otimes \vartheta_{\alg{R}}).
\end{equation} 
Moreover, $\ell$ satisfies the property $\ell\circ \ell=1_{\alg{Q}\otimes \alg{R}}$ --- i.e.\ $(\alg{Q}\otimes \alg{R}, m\otimes n,\ell)$ 
is an involutive quantale, and the following diagram is commutative\textup:
\stepcounter{num}
\begin{equation} \label{neweq2.4}
\begin{tikzcd}[column sep=25pt,row sep=14 pt
]
\alg{Q}\otimes \alg{R}\arrow{r}{\ell}&\alg{Q}\otimes \alg{R}\arrow{r}{\ell}&\alg{Q}\otimes \alg{R}\\
\alg{Q}\arrow{u}{j_{\alg{Q}}}\arrow{r}{\vartheta_{\alg{Q}}}&\alg{R}\arrow{u}{j_{\alg{R}}}\arrow{r}{\vartheta_{\alg{R}}}&\alg{Q}\arrow[swap]{u}{j_{\alg{Q}}}
\end{tikzcd}
\end{equation}
 If $(\alg{Q},m)$ is left-sided, $(\alg{R},n)$ is right-sided, and both quantales are semi-unital, then the involution $\ell$ on $\alg{Q}\otimes \alg{R}$ is uniquely determined by the commutativity of the diagram, and an element $p\in \alg{Q}\otimes \alg{R}$ is a hermitian prime element of $(\alg{Q}\otimes \alg{R}, m\otimes n)$ if and only if there exist a prime element $q$ in $(\alg{Q},m)$ such that $p=(q\otimes \top)\vee (\top\otimes \vartheta_{\alg{Q}}(q))$.
\end{theorem}

\begin{proof} To show that $\ell$ defined by (\ref{neweq2.3}) is an involution on $(\alg{Q}\otimes \alg{R}, m\otimes n)$, we recall that every tensor of $\alg{Q}\otimes \alg{R}$ is an appropriate join of elementary tensors. 
Hence, it is sufficient to verify the following properties:
\begin{enumerate}[label=\textup{(\roman*)},leftmargin=20pt,labelwidth=10pt,itemindent=0pt,labelsep=5pt,topsep=5pt,itemsep=3pt]
\item $\ell(\ell(\alpha\otimes \beta))=\alpha\otimes \beta$ for each $\alpha\in \alg{Q}$ and $\beta\in \alg{R}.$
\item $\ell((\alpha_1\otimes \beta_1)\star (\alpha_2\otimes \beta_2))=\ell(\alpha_2\otimes \beta_2)\star \ell(\alpha_1\otimes \beta_1)$ for each $\alpha_1,\alpha_2\in \alg{Q}$ and $\beta_1,\beta_2\in \alg{R}.$
\end{enumerate}
For property (i), we have:
\[\ell(\ell(\alpha\otimes \beta))= \ell\bigl(\vartheta_{\alg{R}}(\beta)\otimes \vartheta_{\alg{Q}}(\alpha)\bigr)=\vartheta_{\alg{R}}(\vartheta_{\alg{Q}}(\alpha))\otimes \vartheta_{\alg{Q}}(\vartheta_{\alg{R}}(\beta))=\alpha\otimes \beta.\]
Since $\vartheta_{\alg{Q}}$ and $\vartheta_{\alg{R}}$ are quantale anti-homomorphisms, we obtain:
\begin{align*}
\ell((\alpha_1\otimes\beta_1)\star (\alpha_2\otimes \beta_2))&=\ell((\alpha_1\ast\alpha_2)\otimes (\beta_1\ast \beta_2))=\vartheta_{\alg{R}}(\beta_1\ast \beta_2)\otimes \vartheta_{\alg{Q}}(\alpha_1\ast \alpha_2)\\
&=(\vartheta_{\alg{R}}(\beta_2)\ast \vartheta_{\alg{R}}(\beta_1))\otimes (\vartheta_{\alg{Q}}(\alpha_2)\ast \vartheta_{\alg{Q}}(\alpha_1))\\
&= (\vartheta_{\alg{R}}(\beta_2)\otimes \vartheta_{\alg{Q}}(\alpha_2))\star (\vartheta_{\alg{R}}(\beta_1)\otimes \vartheta_{\alg{Q}}(\alpha_1))= \ell(\alpha_2\otimes \beta_2)\star \ell(\alpha_1\otimes \beta_1).
\end{align*}
Hence (ii) is verified.\\[1mm]
Since (\ref{neweq2.2}) implies $\top=\vartheta_{\alg{Q}}(\vartheta_{\alg{R}}(\top))\le \vartheta_{\alg{Q}}(\top)$ and $\top=\vartheta_{\alg{R}}(\vartheta_{\alg{Q}}(\top))\le \vartheta_{\alg{R}}(\top)$ ---
i.e.\  $\vartheta_{\alg{Q}}$ and $\vartheta_{\alg{R}}$ preserve the respective universal upper bounds, we obtain for all 
$\alpha\in \alg{Q}$ 
\[\ell(j_{\alg{Q}}(\alpha))=\ell(\alpha\otimes \top)= \top\otimes \vartheta_{\alg{Q}}(\alpha)= (j_{\alg{R}}\circ \vartheta_{\alg{Q}})(\alpha),\]
Hence, $\ell\circ j_{\alg{Q}}=j_{\alg{R}}\circ\vartheta_{\alg{Q}}$ follows.
Analogously, we verify $\ell\circ j_{\alg{R}}=j_{\alg{Q}}\circ\vartheta_{\alg{R}}$, and so the diagram (\ref{neweq2.4}) is commutative.\\[1mm]
Finally, if $(\alg{Q},m)$ is left-sided and semi-unital, it follows that $\top$ is a left-unit in $\alg{Q}$; respectively, if $(\alg{R},n)$ is right-sided and semi-unital, then $\top$ is a right-unit in $\alg{R}$.  
Consequently the commutativity of the diagram (\ref{neweq2.4}) implies:
\begin{align*}\ell(\alpha\otimes \beta)&= \ell((\top\ast \alpha)\otimes (\beta\ast \top))=\ell(j_{\alg{R}}(\beta)\star j_{\alg{Q}}(\alpha))= \ell(j_{\alg{Q}}(\alpha))\star \ell(j_{\alg{R}}(\beta))\\
&=j_{\alg{R}}(\vartheta_{\alg{Q}}(\alpha))\star 
j_{\alg{Q}}(\vartheta_{\alg{R}}(\beta))=(\top\otimes \vartheta_{\alg{Q}}(\alpha))\star (\vartheta_{\alg{R}}(\beta)\otimes \top)=\vartheta_{\alg{R}}(\beta)\otimes \vartheta_{\alg{Q}}(\alpha)
\end{align*}
for each $\alpha\in \alg{Q}$ and $\beta\in \alg{R}$.
Again, since every tensor of $\alg{Q}\otimes \alg{R}$ is a join of elementary tensors, the formula (\ref{neweq2.3}) follows. 
Finally, it follows from Theorem~\ref{newtheorem1} that $p\in \sigma(\alg{Q}\otimes \alg{R})$ if and only if there exist $q\in\sigma(\alg{Q})$ and  $r\in\sigma(\alg{R})$ such that $p=(q\otimes \top)\vee (\top\otimes r)$. Since
\[\ell(p)=\ell(q\otimes \top)\vee \ell(\top\otimes r)=(\top\otimes \vartheta_{\alg{Q}}(q))\vee(\vartheta_{\alg{R}}(r)\otimes \top),\]
it follows that $p$ is hermitian if and only if $\vartheta_{\alg{Q}}(q)=r$ and $\vartheta_{\alg{R}}(r)=q$.
\end{proof}

\begin{remark*} 
Let $(\alg{Q},m)$ be a balanced quantale and $(\alg{Q},m^{op})$ be the opposite quantale. 
Then the identity $(\alg{Q},m)\xrightarrow{\,1_{\alg{Q}}\,} (\alg{Q},m^{op})$ is a quantale 
anti-homomorphism 
such that condition (\ref{neweq2.2}) is satisfied. 
We conclude from Theorem~\ref{newtheorem3} that $(\alg{Q}\otimes \alg{Q}, m\otimes  m^{op}, c_{\alg{Q}\alg{Q}})$ is an involutive quantale. 
This observation is a generalization of \cite[Cor.~2.13]{GHK22}. Since $(\alg{Q},m)$ is 
 balanced, the canonical embedding $\alg{Q}\xrightarrow{\, j_{\alg{Q}}\,} \alg{Q}\otimes \alg{Q}$ is a strong quantale homomorphism. 
 Hence, every balanced quantale can be embedded
  into an involutive quantale in the sense of \cat{Quant}.\\
 If $\iota$ is an idempotent element of $\alg{Q}$ with $\iota\not\in\set{\bot,\top}$, then $\alpha \longmapsto \alpha\otimes \iota$\ is still an embedding from $(\alg{Q},m)$ into  $\bigl(\alg{Q}\otimes \alg{Q}, m\otimes m^{op}, c_{\alg{Q}\alg{Q}}\bigr)$, but not a strong quantale homomorphism.
\end{remark*}

\begin{corollary} \label{newcorollary2} Let $\alg{Q}$ be an involutive and semi-unital quantale. Further, let $(\mathds{L}(\alg{Q}),m_l)$ and $(\mathds{R}(\alg{Q}),m_r)$ be the subquantales of all left-sided  elements and right-sided elements of $\alg{Q}$, respectively.
 Then there exists a unique involution $\ell$ on the tensor product $(\mathds{L}(\alg{Q})\otimes \mathds{R}(\alg{Q}),m_l\otimes m_r)$ satisfying the following property\textup:
\stepcounter{num}
\begin{equation}\label{neweq2.5}  
\ell(\alpha \otimes \beta)=\beta^{\prime} \otimes \alpha^{\prime}, \qquad \alpha \in \mathds{L}(\alg{Q}),\, \beta \in \mathds{R}(\alg{Q}).
\end{equation}
\end{corollary}

\begin{proof} The restriction of the involution in $\alg{Q}$ 
induces a pair of anti-homomorphisms between $\mathds{L}(\alg{Q})$ and $\mathds{R}(\alg{Q})$ satisfying (\ref{neweq2.2}). Since $\mathds{L}(\alg{Q})$ and $\mathds{R}(\alg{Q})$ are semi-unital, we conclude from Theorem~\ref{newtheorem3} that there exists a unique involution $\ell$ on the tensor product $\mathds{L}(\alg{Q})\otimes \mathds{R}(\alg{Q})$ satisfying (\ref{neweq2.5}).  
\end{proof}

\subsection{Tensorially involutive quantale associated with an involutive and semi-unital quantale}
\label{subsec:2.1}
The previous Corollary~\ref{newcorollary2} suggests to introduce the following terminology. Let $\alg{Q}$ be an involutive and semi-unital quantale. Then the involutive quantale $(\mathds{L}(\alg{Q})\otimes \mathds{R}(\alg{Q}),m_l\otimes m_r,\ell)$ satisfying (\ref{neweq2.5}) is called the \emph{tensorially involutive} quantale associated with $\alg{Q}$. 
\begin{remark}\label{newremark1}
The tensorially involutive quantale $(\mathds{L}(\alg{Q})\otimes \mathds{R}(\alg{Q}),m_l\otimes m_r,\ell)$ of $\alg{Q}$ is always semi-unital and semi-integral. Moreover, it is pre-idempotent if and only if $\mathds{L}(\alg{Q})$ is idempotent and it is bisymmetric if and only if $\mathds{L}(\alg{Q})$ is bisymmetric.
\end{remark}

The next result describes the hermitian spectrum of $(\mathds{L}(\alg{Q})\otimes \mathds{R}(\alg{Q}),m_l\otimes m_r,\ell)$ in terms of the spectrum of $\mathds{L}(\alg{Q})$.
\begin{corollary}\label{newcorollary3} Let $\alg{Q}$ be an involutive and semi-unital quantale, and let $(\mathds{L}(\alg{Q})\otimes \mathds{R}(\alg{Q}),m_l\otimes m_r,\ell)$ be its tensorially involutive quantale. Then there exists a bijective map 
\[{\sigma(\mathds{L}(\alg{Q}))\xrightarrow{\,\psi_l\,}\sigma_h(\mathds{L}(\alg{Q})\otimes \mathds{R}(\alg{Q}))}\]
 given by
 $\psi_l(q)=(q\otimes \top)\vee (\top\otimes q^{\prime})$ for each $q\in \sigma(\mathds{L}(\alg{Q}))$.
 \end{corollary}
 
 \begin{proof} 
 Since every hermitian prime element $p$ of  $(\mathds{L}(\alg{Q})\otimes \mathds{R}(\alg{Q}),m_l\otimes m_r,\ell)$ has the form $p=(q\otimes \top)\vee (\top\otimes q^{\prime})$ with $q\in \sigma(\mathds{L}(\alg{Q}))$ (cf.\ Theorem~\ref{newtheorem3}), the map $\psi_l$ is surjective. The injectivity of this map follows from the subsequent equivalence for all $\alpha_1,\alpha_2\in \mathds{L}(\alg{Q})$ and $\beta_1,\beta_2\in \mathds{R}(\alg{Q})$:
 \[(\alpha_1\otimes \top) \vee (\top\otimes \beta_1)=(\alpha_2\otimes \top) \vee (\top\otimes \beta_2) \iff (\alpha_1=\alpha_2\text{ and }\beta_1=\beta_2). \qedhere\]
 \end{proof}
 
\subsection{The quantization of $\boldsymbol{2}$}
\label{subsec:2.2}
Let $\alg{Q}$ be an involutive and semi-unital quantale. Obviously, if $\mathds{L}(\alg{Q})=\mathds{R}(\alg{Q})=\mathds{I}(\alg{Q})=\set{\bot,\top}$, then the tensorially involutive quantale associated with $\alg{Q}$ is the Boolean \donotbreakdash{$2$}chain $\boldsymbol{2}$. 
We now present the simplest example of a tensorially involutive quantale associated with a quantale $\alg{Q}$ that has some left-sided element being different from $\bot$ and $\top$. 

First, we recall  that on the \donotbreakdash{$3$-}chain $C_3=\set{\bot,a,\top}$ with $\bot < a < \top$, 
there exist exactly two semi-unital and non-commutative quantales, namely $C_3^{\ell}=(C_3,m_{\ell})$ 
and $C_3^r=(C_3,m_r)$, where $a\in C_3^{\ell}$ is left-sided, but not right-sided and $a\in C_3^r$ is right-sided, 
but not left-sided. In particular, the multiplication tables of the quantale multiplications $\ast_{\ell}$ in $C_3^{\ell}$ and $\ast_r$ in $C_3^r$ are given by: 
 \begin{center}
{\footnotesize\begin{tabular}{c||c|c}
$\ast_{\ell}$ & $a$ & $\top$\\
\hline\hline
$a$ &  $a$ & $\top$ \\\hline
$\top$ & $a$ &  $\top$ \\
\end{tabular}
\qquad\qquad\qquad
\begin{tabular}{c||c|c}
$\ast_r$ & $a$ &  $\top$\\
\hline\hline
$a$ &  $a$  & $a$ \\\hline
$\top$ & $\top$  &$\top$ \\
\end{tabular}}
\end{center}

Now we consider  an involutive and semi-unital quantale $\alg{Q}$ such that $\mathds{I}(\alg{Q})=\set{\bot,\top}$ 
and the subquantales $\mathds{L}(\alg{Q})$ and $\mathds{R}(\alg{Q})$ are isomorphic to $C_3^{\ell}$ and $C_3^r$.
Then the tensorially involutive quantale associated with $\alg{Q}$ is isomorphic to the tensor product $C_3^{\ell}\otimes C_3^r$. 
Since  $C_3^{\ell}$ and $ C_3^r$ are idempotent and therefore bisymmetric, the quantale $C_3^{\ell}\otimes C_3^r$ is involutive, semi-unital, semi-integral, pre-idempotent and bisymmetric (cf. Remark~\ref{newremark1}).
To simplify the notation we will denote $C_3^{\ell}\otimes C_3^r$ by $\alg{Q}_{\boldsymbol{2}}$, and call it the \emph{quantization of $\boldsymbol{2}$}.

 Since $\alg{Q}_{\boldsymbol{2}}$ consists of six elements, we introduce the following abbreviations:
\[a_{\ell}:= a\otimes \top,\quad a_r:= \top\otimes a,\quad b:=a\otimes a,\quad c:= a_{\ell}\vee a_r.\]    
 Then the action of the involution on $\alg{Q}_{\boldsymbol{2}}$ is given by $\top^{\prime}=\top$, $c^{\prime}=c$, $a_{\ell}^{\prime}=a_r$, $a_r^{\prime}=a_{\ell}$, $b^{\prime}=b$, and $\bot^{\prime}=\bot$, and the Hasse diagram  and the multiplication table of $\alg{Q}_{\boldsymbol{2}}$ have the form: 
 \begin{center}
\renewcommand\arraystretch{1.3}
{\footnotesize$\begin{tikzcd}[column sep=5pt,row sep=5pt]
&\top \arrow[-]{d} &\\
&c\arrow[-]{dl}\arrow[-]{dr}\\
{a_{\ell}}&&{a_r}\\
&b\arrow[-]{ul}\arrow[-]{ur}&\\
&\bot \arrow[-]{u} 
\end{tikzcd}$}
\qquad\qquad
{\footnotesize
\begin{tabular}{c||c|c|c|c|c}
$\star$  & $b$ & $a_{\ell}$ & $a_r$ & $c$ & $\top$\\
\hline\hline
$b$ &  $b$ & $b$ & $a_r$ & $a_r$ & $a_r$\\\hline
$a_{\ell}$ & $a_{\ell}$ & $a_{\ell}$ & $\top$ & $\top$ &$\top$\\\hline
$a_r$ & $b$ & $b$ & $a_r$ &$a_r$ & $a_r$\\\hline
$c$ &$a_{\ell}$ & $a_{\ell}$ & $\top$ & $\top$ & $\top$\\\hline
$\top$ &  $a_{\ell}$ & $a_{\ell}$ & $\top$ & $\top$ & $\top$\\
\end{tabular}}
\end{center} 

It is interesting to see that the right quantale multiplication in $\alg{Q}_{\boldsymbol{2}}$ satisfies the following properties:
\stepcounter{num}
\begin{equation}\label{neweq2.6} \underline{\mbox{\,\,\,\,}}\star b=\underline{\mbox{\,\,\,\,}}\star a_{\ell}\quad \text{and}\quad\underline{\mbox{\,\,\,\,}}\star a_r=\underline{\mbox{\,\,\,\,}}\star c=\underline{\mbox{\,\,\,\,}}\star \top.
\end{equation}
Since $\alg{Q}_{\boldsymbol{2}}$ is involutive, we also have the respective properties for the left quantale multiplication:
\[b\star \underline{\mbox{\,\,\,\,}}=a_r\star\underline{\mbox{\,\,\,\,}}\quad \text{and}\quad a_{\ell}\star\underline{\mbox{\,\,\,\,}}=c\star\underline{\mbox{\,\,\,\,}}=\top\star\underline{\mbox{\,\,\,\,}}.\]

Referring to Lemma~\ref{newlemma1}, the quantale multiplication in $\alg{Q}_{\boldsymbol{2}}$ satisfies the subsequent
important property for all $\alpha,\beta \in\alg{Q}_{\boldsymbol{2}}$:
\stepcounter{num}
\begin{equation}\label{neweq2.7}  {\alpha\star \gamma\star \beta=\alpha\star \beta\qquad\gamma \in\alg{Q}_{\boldsymbol{2}}\setminus\{ \bot\}}.
\end{equation}
Finally, the element $c=a_{\ell}\vee a_r$ is the largest prime element of $\alg{Q}_{\boldsymbol{2}}$. Obviously, $c$ is also a hermitian prime element.

The importance of the quantization of $\boldsymbol{2}$ is explained in the next remark.

\begin{remark}\label{newremark3}
Given a semi-unital quantale $\alg{Q}$, every strong quantale homomorphism $\alg{Q}\xrightarrow{\,h\,} \alg{Q}_{\boldsymbol{2}}$ can be identified with a prime element $p$ of $\alg{Q}$ and vice versa (cf.\ \cite{Ho15B} and \cite[{Thm.~2.5.15}]{EGHK}). This bijection can be expressed as follows.
 If $h^{\vdash}$ is the right adjoint map of $h$, then $p=h^{\vdash}(c)$, where $c$ is the largest prime element of $\alg{Q}_{\boldsymbol{2}}$ (cf.\ \cite[Lem.~2.5.14]{EGHK}). 
 On the other hand, if $\alg{Q}$ is a semi-unital quantale and $p$ a prime element of $\alg{Q}$, 
 then $p$ induces a quantale homomorphism $\alg{Q}\xrightarrow{\,h_p\,} \alg{Q}_{\boldsymbol{2}}$ by:
 \stepcounter{num}
\begin{equation}\label{neweq2.8}
h_p(\alpha)=\begin{cases} 
\bot,&\top\ast \alpha\ast \top\le p,\\
b,& \top\ast \alpha\ast \top\not\le p,\ \alpha\ast \top\le p\text{ and } \top \ast \alpha\le p,\\
a_{\ell},& \alpha\ast \top\not\le p\text{ and } \top \ast \alpha\le p,\\
a_r,& \alpha\ast \top\le p\text{ and } \top \ast \alpha\not\le p,\\
c, & \alpha\le p,\ \alpha\ast \top\not\le p\text{ and } \top \ast \alpha\not\le p,\\
\top, & \alpha\not\le p,
\end{cases} 
\qquad \alpha\in \alg{Q}.
\end{equation}
Obviously, $h_p^{\vdash}(c)=p$. If $\alg{Q}$ is also involutive, then $h_p$ is involutive if and only if $p$ is hermitian.
\end{remark}

 This construction has a generalization to \emph{arbitrary} quantales (cf.\ \cite[Rem.~4.1]{Ho15B}) since every 
strong quantale homomorphism $\alg{Q}\xrightarrow{\,h\,} \alg{Q}_{\boldsymbol{2}}$ can be identified with the strong quantale homomorphism $\widehat{\alg{Q}}\xrightarrow{\,\widehat{h}\,} \alg{Q}_{\boldsymbol{2}}$ given by $\widehat{h}(\widehat{\top})=\top$ and $\widehat{h}(\alpha)=h(\alpha)$ for all $\alpha\in\alg{Q}$ (cf. Remarks~\ref{newremarks1}\,(2)). 

 For the convenience of the reader we recall the construction here and summarize the situation
as follows:
\begin{proposition} \label{newproposition3}
Let $\alg{Q}$ be an arbitrary quantale. 
 \begin{enumerate}[label=\textup{(\roman*)},leftmargin=20pt,labelwidth=10pt,itemindent=0pt,labelsep=5pt,topsep=5pt,itemsep=3pt]
\item If $\alg{Q}\xrightarrow{\, h\,} \alg{Q}_{\boldsymbol{2}}$ is a strong quantale homomorphism, then $h^{\vdash}(c)$ is strongly prime in $\alg{Q}$, where $c$ is the largest prime element in $\alg{Q}_{\boldsymbol{2}}$.
\item  If $p\in \alg{Q}$ is strongly prime, then there exists a unique strong quantale homomorphism ${\alg{Q}\xrightarrow{\, h\,} \alg{Q}_{\boldsymbol{2}}}$
satisfying the condition $p=h^{\vdash}(c)$. 
\end{enumerate}
Moreover, if $\alg{Q}$ is involutive then, a strong quantale homomorphism $\alg{Q}\xrightarrow{\,h\,} \alg{Q}_{\boldsymbol{2}}$ is involutive if and only if its associated strongly prime element $p=h^{\vdash}(c)$ is hermitian.
\end{proposition}

\begin{proof} The assertion (i) is \cite[Lem.~2.5.14]{EGHK}. To prove (ii) we proceed as follows. Every strongly prime element $p$ in $\alg{Q}$ is a prime element in the semi-unitalization $\widehat{\alg{Q}}$ of $\alg{Q}$. Since  $\widehat{\alg{Q}}$ is semi-unital, $p$ can be identified with a unique strong quantale homomorphism 
$\widehat{\alg{Q}}\xrightarrow{\, \widehat{h}_p\,} \alg{Q}_{\boldsymbol{2}}$ such that $p=(\widehat{h}_p)^{\vdash}(c)$. Since $\widehat{h}_p(\top)=\top$ (cf.\ \ref{neweq2.8}), the restriction of $\widehat{h}_p$ to $\alg{Q}$ is a strong homomorphism $\alg{Q}\xrightarrow{\,h\,} \alg{Q}_{\boldsymbol{2}}$   satisfying $h^{\vdash}(c)=p$. The uniqueness of $h$ follows from the uniqueness of $\widehat{h}_p$.

Since $c$ is a hermitian prime element of $\alg{Q}_{\boldsymbol{2}}$, $h^{\vdash}(c)$ is hermitian  for every involutive quantale homomorphism.  Hence, the final assertion follows from (\ref{neweq2.8}).
\end{proof}

\subsection{Quantic frames}
\label{subsec:2.3}
 Let $\cat{BSQuant}$ be the category of balanced and bisym\-me\-tric quantales and strong quantale homomorphisms. We  now refer to \cite[Def.~2.5.13]{EGHK} and recall that quantic frames are special objects of $\cat{BSQuant}$. More precisely, an involutive, bisymmetric and semi-unital quantale $\alg{Q}=(\alg{Q},m,\ell)$ is called a \emph{quantic frame} if $\alg{Q}$ satisfies the following additional properties:
\begin{enumerate}[label=\textup{(\Alph*)},leftmargin=20pt,labelwidth=10pt,itemindent=0pt,labelsep=5pt,topsep=5pt,itemsep=3pt]
\item \label{A} Every left-sided (resp. right-sided) element is idempotent.
\item \label{B} Every two-sided element is hermitian. 
\item \label{C} Let 
${\mathds{I}(\alg{Q})\xhookrightarrow{\, q_{\mathds{L}(\alg{Q})}\,} \mathds{L}(\alg{Q})}$, $\mathds{I}(\alg{Q})\xhookrightarrow{\, q_{\mathds{R}(\alg{Q})}\,} \mathds{R}(\alg{Q})$, $\mathds{L}(\alg{Q}) \xhookrightarrow{\, \iota_{\mathds{L}(\alg{Q})}\,} \alg{Q}$ and $\mathds{R}(\alg{Q}) \xhookrightarrow{\, \iota_{\mathds{R}(\alg{Q})}\,} \alg{Q}$ be the inclusion arrows. 
Then the commutative diagram
\stepcounter{num}
\begin{equation}\label{neweq2.9}
\begin{tikzcd}[column sep=25pt,row sep=5pt]
&\alg{Q}&\\
\mathds{L}(\alg{Q})\arrow[hookrightarrow]{ru}{\iota_{\mathds{L}(\alg{Q})}}&&\mathds{R}(\alg{Q})\arrow[hookrightarrow,swap]{lu}{\iota_{\mathds{L}(\alg{Q})}}\\
&\mathds{I}(\alg{Q})\arrow[hookrightarrow]{lu}{q_{\mathds{L}(\alg{Q})}}
\arrow[hookrightarrow,swap]{ru}{q_{\mathds{R}(\alg{Q})}}
\end{tikzcd}
\end{equation}
\noindent
is a pushout square in $\cat{BSQuant}$.
\end{enumerate}

Before we proceed, we first recall a characterization of condition \ref{C} by a well known coequalizer condition. Since the tensorially involutive quantale $\mathds{L}(\alg{Q})\otimes \mathds{R}(\alg{Q})$ associated with a semi-unital, bisymmetric and involutive quantale $\alg{Q}$ is the coproduct of $\mathds{L}(\alg{Q})$ and $\mathds{R}(\alg{Q})$ in $\cat{BSQuant}$ with the coprojections  $\mathds{L}(\alg{Q})\xrightarrow{\, j_{\mathds{L}(\alg{Q})}\,} \mathds{L}(\alg{Q})\otimes\mathds{R}(\alg{Q})$ and $\mathds{R}(\alg{Q})\xrightarrow{\,j_{\mathds{R}(\alg{Q})}\,} \mathds{L}(\alg{Q})\otimes \mathds{R}(\alg{Q})$, for every pair of strong quantale homomorphisms $\mathds{L}(\alg{Q})\xrightarrow{\,f\,} \alg{Q}$ and $\mathds{R}(\alg{Q})\xrightarrow{\,g\,} \alg{Q}$ there exists a unique arrow   $\mathds{L}(\alg{Q})\otimes \mathds{R}(\alg{Q}) \xrightarrow{\,f\sqcup g\,} \alg{Q}$ making the following diagram commutative:
\[
\begin{tikzcd}[column sep=25pt,row sep=15pt]
&\alg{Q}&\\
\mathds{L}(\alg{Q})\arrow[r,"j_{\mathds{L}(\alg{Q})}" near end]\arrow{ru}{f}&\mathds{L}(\alg{Q})\otimes \mathds{R}(\alg{Q})\arrow{u}{f\sqcup g}&\mathds{R}(\alg{Q})\arrow[swap]{lu}{g}\arrow[l,swap,"j_{\mathds{R}(\alg{Q})}" near end]
\end{tikzcd}
\]
Then the diagram in (\ref{neweq2.9}) is a pushout square in $\cat{BSQuant}$ if and only if $\iota_{\mathds{L}(\alg{Q})}\sqcup \iota_{\mathds{R}(\alg{Q})}$ is the coequalizer of  $\mathds{I}(\alg{Q}) \xleftleftarrows[\, j_{\mathds{R}(\alg{Q})}\circ q_{\mathds{R}(\alg{Q})}\,]{\,j_{\mathds{L}(\alg{Q})}\circ q_{\mathds{L}(\alg{Q})}\,} \mathds{L}(\alg{Q})\otimes \mathds{R}(\alg{Q})$ in $\cat{BSQuant}$. Since coequalizers in $\cat{BSQuant}$ are also coequalizers in $\cat{Quant}$, $\mathds{L}(\alg{Q})\otimes \mathds{R}(\alg{Q}) \xrightarrow{\,\iota_{\mathds{L}(\alg{Q})}\sqcup \iota_{\mathds{R}(\alg{Q})}\,} \alg{Q}$ is a regular epimorphism  in $\cat{Quant}$ and consequently a surjective quantale homomorphism (cf.\ \cite[p.\ 91]{EGHK}).

\begin{remark}\label{newremark4} In what follows, if $\alg{Q}$ is a quantic frame, we will always denote  the coequalizer   
\[\mathds{I}(\alg{Q}) \xleftleftarrows[\, j_{\mathds{R}(\alg{Q})}\circ q_{\mathds{R}(\alg{Q})}\,]{\,j_{\mathds{L}(\alg{Q})}\circ q_{\mathds{L}(\alg{Q})}\,} \mathds{L}(\alg{Q})\otimes \mathds{R}(\alg{Q})\xrightarrow{\,\pi\,} \alg{Q}\]
by $\pi:=\iota_{\mathds{L}(\alg{Q})}\sqcup \iota_{\mathds{R}(\alg{Q})}$.

Let $\alg{Q}$ an involutive and semi-unital such that $\mathds{I}(\alg{Q})= \set{\bot,\top}$ --- i.e.\ $\alg{Q}$ is a factor.
\begin{enumerate}[label=\textup{(\alph*)},leftmargin=20pt,labelwidth=10pt,itemindent=0pt,labelsep=5pt,topsep=5pt,itemsep=3pt]
\item If $\alg{Q}$ is a quantic frame, then $\pi$ is an isomorphism --- i.e. $\alg{Q}$ is isomorphic to its tensorially involutive quantale. 
\item If $\mathds{L}(\alg{Q})$ is idempotent and therefore bisymmetric, then $\mathds{L}(\alg{Q})\otimes \mathds{R}(\alg{Q})$ is a quantic frame (cf. Remark~\ref{newremark1}). Consequently, $\alg{Q}_{\boldsymbol{2}}$ is the simplest \emph{non-trivial} example of a quantic frame. 
\end{enumerate}
\end{remark}

Finally, we recall some general properties of quantic frames. First, we emphasize again that every quantic frame is bisymmetric.  Further, it follows from \ref{A} and \ref{C} that every quantic frame is semi-unital, semi-integral and pre-idempotent. A quantic frame is commutative if and only if it is a traditional frame (cf.\ \cite[p.~149]{EGHK} and \cite{Johnstone}).
\begin{proposition} \label{newproposition4} Let $\alg{Q}$ be a quantic frame.
\begin{enumerate}[label=\textup{(\alph*)},leftmargin=20pt,labelwidth=10pt,itemindent=0pt,labelsep=5pt,topsep=5pt,itemsep=3pt]
\item If $\mathds{I}(\alg{Q})$ does not have zero divisors, then for $\alpha\in \mathds{L}(\alg{Q})\setminus\{\bot\} $ and for $\beta\in\mathds{R}(\alg{Q})\setminus \{\bot\}$ the relation $\pi(\alpha \otimes \beta)\neq \bot$ holds.
\item $\pi$ is involutive.
\item The restriction of the right adjoint map $\pi^{\vdash}$ of $\pi$ to $\sigma_h(\alg{Q})$ is a bijective  map from $\sigma_h(\alg{Q})$ to $\sigma_h(\mathds{L}(\alg{Q})\otimes\mathds{R}(\alg{Q}))$.
\item The restriction of $\pi$ to $\sigma_h(\mathds{L}(\alg{Q})\otimes\mathds{R}(\alg{Q}))$ is a bijective map from $\sigma_h(\mathds{L}(\alg{Q})\otimes\mathds{R}(\alg{Q}))$ to $\sigma_h(\alg{Q})$.
\end{enumerate}
\end{proposition}

\begin{proof} (a) We choose $\alpha\in \mathds{L}(\alg{Q})$ and  $\beta\in\mathds{R}(\alg{Q})$ with $\pi(\alpha\otimes \beta)=\bot$. Then $\alpha\ast \top$ and $\top\ast \beta$ are two-sided and since $\iota_{\mathds{L}(\alg{Q})} =\pi \circ  j_{\mathds{L}(\alg{Q})}$ and $\iota_{\mathds{R}(\alg{Q})} =\pi \circ  j_{\mathds{R}(\alg{Q})}$, we obtain:
\begin{align*}\top\ast  \beta\ast\alpha\ast \top&= (\pi\circ j_{\mathds{R}(\alg{Q})})(\top\ast \beta)\ast(\pi\circ j_{\mathds{L}(\alg{Q})})(\alpha\ast \top)
=\pi((\top \otimes (\top\ast  \beta))\star ((\alpha\ast \top)\otimes \top))
\\
&= \pi((\top\otimes \top)\star(\alpha \otimes \beta)\star(\top\otimes \top))= \pi(\top\otimes \top)\ast \pi(\alpha \otimes \beta)\ast \pi(\top\otimes \top)=\bot,
\end{align*}
 hence $\top\ast \beta=\bot$ or $\alpha \ast \top=\bot$ holds. Since $\alg{Q}$ is semi-unital, $\alpha=\bot$ or $\beta=\bot$ follows.\\[2mm]
 (b) Let $\alpha\in \mathds{L}(\alg{Q})$ and $\beta\in\mathds{R}(\alg{Q})$. Since $\pi(\alpha\otimes \top)=\alpha$, $\pi(\top\otimes\beta)=\beta$, $\pi(\beta^{\prime}\otimes \top)=\beta^{\prime}$, $\pi(\top\otimes\alpha^{\prime})=\alpha^{\prime}$, and $\mathds{L}(\alg{Q})\otimes \mathds{R}(\alg{Q})$ is the tensorially involutive quantale associated with $\alg{Q}$, we refer to Corollary~\ref{newcorollary2} and obtain:
\begin{align*}\pi(\alpha\otimes \beta)^{\prime}&=\pi((\top\otimes \beta)\ast (\alpha\otimes \top))^{\prime}=(\pi(\top\otimes \beta)\ast \pi(\alpha\otimes \top))^{\prime}= (\beta\ast \alpha)^{\prime}=\alpha^{\prime}\ast \beta^{\prime}\\
&=\pi(\top\otimes \alpha^{\prime})\ast \pi(\beta^{\prime}\otimes \top)=\pi((\top\otimes \alpha^{\prime})\ast (\beta^{\prime}\otimes \top))=\pi(\beta^{\prime}\otimes \alpha^{\prime})=\pi(\ell(\alpha\otimes \beta)).
\end{align*}
(c) Since $\pi$ is surjective and involutive (cf.\ (b)), the right adjoint map ${\alg{Q}}\xrightarrow{\, \pi^{\vdash}\,} \mathds{L}(\alg{Q})\otimes \mathds{R}(\alg{Q})$ of $\pi$ is injective and also involutive. Hence $\pi^{\vdash}(p)\in \sigma_h(\mathds{L}(\alg{Q})\otimes \mathds{R}(\alg{Q}))$ for each $p\in\sigma_h(\alg{Q})$. 
Now let $p$ be a hermitian prime element of $\mathds{L}(\alg{Q})\otimes \mathds{R}(\alg{Q})$. Then we conclude from Remark~\ref{newremark3} that $p$ can be identified with  an involutive and strong quantale homomorphism $\mathds{L}(\alg{Q})\otimes \mathds{R}(\alg{Q})\xrightarrow{\,h_p\,} \alg{Q}_{\boldsymbol{2}}$ satisfying $h_p^{\vdash}(c)=p$.  Since every two-sided element of $\alg{Q}$ is hermitian (cf.\ \ref{B}), the following equivalence follows from Corollary~\ref{newcorollary2} for all $\alpha\in \mathds{I}(\alg{Q})$:
\[\alpha \otimes \top \le p\,\,\iff\,\, \top \otimes \alpha= \top \otimes \alpha^{\prime}=(\alpha \otimes \top)^{\prime}\le p^{\prime}=p.\]
Since $\mathds{L}(\alg{Q})\otimes \mathds{R}(\alg{Q})$ is semi-unital, the previous equivalence means (cf.\ (\ref{neweq2.8})):
 \[h_p\circ j_{\mathds{L}(\alg{Q})}\circ q_{\mathds{L}(\alg{Q})}=h_p\circ j_{\mathds{R}(\alg{Q})}\circ q_{\mathds{R}(\alg{Q})}.\] 
 Hence, $h_p$  factors through $\mathds{L}(\alg{Q}) \otimes \mathds{R}(\alg{Q}) \xrightarrow{\,\pi\,}\alg{Q}$, and consequently, there exists a unique strong quantale homomorphism $\alg{Q}\xrightarrow{\, h_0\,} \alg{Q}_{\boldsymbol{2}}$ such that $h_p=h_0\circ \pi$. Since $h_p$ and $\pi$ are involutive, so is $h_0$, and $p_0:=h_0^{\vdash}(c)\in \sigma_h(\alg{Q})$ holds. Hence, $p=h_p^{\vdash}(c)= \pi^{\vdash}(h_0^{\vdash}(c))=\pi^{\vdash}(p_0)$ follows,
and the image 
of the restriction of $\pi^{\vdash}$ to $\sigma_h(\alg{Q})$ coincides with $\sigma_h(\mathds{L}(\alg{Q})\otimes \mathds{R}(\alg{Q}))$ --- i.e.\ (c)  is verified.\\[2mm]
(d) Being a coequalizer, $\pi$ is necessarily surjective, and the relation $\pi\circ \pi^{\vdash}=1_{\alg{Q}}$ holds. Hence, we conclude from (c) that it is sufficient to show that the restriction of $\pi$ to $\sigma_h(\mathds{L}(\alg{Q})\otimes \mathds{R}(\alg{Q}))$ is injective. Therefore, let $p\in \sigma_h(\mathds{L}(\alg{Q})\otimes \mathds{R}(\alg{Q}))$. With regard to (c) we can choose $r\in \sigma_h(\alg{Q})$ with $p=\pi^{\vdash}(r)$. Then we obtain
$\pi^{\vdash}(\pi(p))=\pi^{\vdash}(r)=p$.
Hence (d) follows.
\end{proof}

\begin{corollary} \label{newcorollary4} Let $\alg{Q}$ be a quantic frame. Then
there exists a bijective map $\sigma(\mathds{L}(\alg{Q})) \xrightarrow{\, \varphi\,} \sigma_h(\alg{Q})$ given by
 $\varphi(q)=q\vee q^{\prime}$ for each $q\in \sigma(\mathds{L}(\alg{Q}))$.
\end{corollary}

\begin{proof} 
Since $\iota_{\mathds{L}(\alg{Q})} =\pi \circ  j_{\mathds{L}(\alg{Q})}$ and $\iota_{\mathds{R}(\alg{Q})} =\pi \circ  j_{\mathds{R}(\alg{Q})}$, the assertion follows from Corollary \ref{newcorollary3} and Proposition~\ref{newproposition4}\,(d).
\end{proof}

\begin{proposition} \label{newproposition5} A quantic frame  $\alg{Q}$ does not have zero divisors if and only if $\mathds{I}(\alg{Q})$ does not have zero divisors.
\end{proposition}

\begin{proof} The necessity is obvious. To verify the sufficiency we choose $\alpha_0,\beta_0\in\alg{Q}$ with $\alpha_0\ast\beta_0=\bot$ and assume $\alpha_0\neq\bot$ and $\beta_0\neq\bot$.  Being a coequalizer, $\pi$ is necessarily surjective, and we now choose elementary tensors $\alpha_1\otimes \alpha_2,\,\beta_1\otimes \beta_2\in \mathds{L}(\alg{Q})\otimes \mathds{R}(\alg{Q})$ such that
\[\bot \neq \pi(\alpha_1\otimes \alpha_2) \le \alpha_0 \quad \text{and}\quad \bot\neq\pi(\beta_1\otimes \beta_2)\le \beta_0\]
hold. Hence $\alpha_1\neq\bot,\,\beta_1\neq\bot,\,\alpha_2\neq \bot$ and $\beta_2\neq\bot$. Since $\mathds{I}(\alg{Q})$ does not have zero divisors, we conclude from 
 Lemma~\ref{newlemma2}\,(a) that $\alpha_1\ast\beta_1\neq \bot$ and $ \alpha_2\ast \beta_2\neq\bot$. Now we apply Proposition~\ref{newproposition4}\,(a) and obtain
$\pi\bigl((\alpha_1\ast\beta_1)\otimes (\alpha_2\ast\beta_2)\bigr) \neq \bot$. Thus
\[\bot\neq \pi\bigl((\alpha_1\ast\beta_1)\otimes (\alpha_2\ast\beta_2)\bigr) =\pi(\alpha_1\otimes \alpha_2) \ast \pi(\beta_1\otimes\beta_2)\le \alpha_0\ast\beta_0\]
follows, which is a contradiction. Hence the assumption is false and the assertion is verified.
\end{proof}

\subsection{The spectrum of a unital $C^*$-algebra}
\label{subsec:2.4}
The next example shows that every unital \donotbreakdash{$C^*$}algebra $A$ gives rise to an involutive quantale playing the role of the spectrum of $A$.
\begin{example}\label{newexample2} Let $A=(A,+,\cdot,\mbox{}^{*}, 1)$ be a unital \donotbreakdash{$C^*$}algebra with unit $1$ and involution $\mbox{}^{*}$.  Following \cite{Mulvey01}, we consider the involutive and unital quantale $\mathds{M}axA$ of all closed linear subspaces of $A$ provided with the following quantale multiplication:
\[M_1\ast M_2= \text{top.\ closure}\bigl(\bigset{ \tsum\limits_{i=1}^n a_i\cdot b_i\mid n\in \mathds{N},\, a_i\in M_1,\, b_i\in M_2}\bigr),\qquad M_1,M_2\in \mathds{M}axA.\]
The involution of  $\mathds{M}axA$ assigns to any closed linear subspace $M$ the adjoint closed linear subspace $M^{\prime}$ determined by
\[M^{\prime}=\set{a^*\in A\mid a\in M},\]
where $\mbox{}^*$  is the involution of the given \donotbreakdash{$C^*$}algebra $A$.

The subquantale $\mathds{L}(\mathds{M}axA)$ of all left-sided elements of $\mathds{M}axA$  coincides with the quantale $\mathds{L}(A)$ of all closed left ideals of $A$, and the restriction of the quantale multiplication $\ast$ to $\mathds{L}(A)$ is the usual \emph{ideal multiplication}. The same applies to right-sided elements of $\mathds{M}axA$ --- i.e.\ $\mathds{R}(\mathds{M}axA)=\mathds{R}(A)$ is the quantale of all closed right ideals of $A$ and again the restriction of $\ast$ is the usual ideal multiplication.

The extension of the quantale multiplication $\ast$ to a semigroup operation in $\cat{Sup}$ is denoted by $m$, and consequently, the extension of the ideal multiplication in $\mathds{L}(A)$ and $\mathds{R}(A)$ is denoted by $m_{\mathds{L}}$ in $\mathds{L}(A)$ and  $m_{\mathds{R}}$ in $\mathds{R}(A)$, respectively.  Since $\mathds{M}axA$  is unital, 
$(\mathds{L}(A),m_{\mathds{L}})$ and $(\mathds{R}(A),m_{\mathds{R}})$ are semi-unital quantales. 
In particular, the universal upper bound $\top=A$ is a left unit in $\mathds{L}(A)$ and a right unit in $\mathds{R}(A)$, and consequently, $(\mathds{L}(A),m_{\mathds{L}})$ and $(\mathds{R}(A),m_{\mathds{R}})$ are semi-integral. Therefore, we can apply Fact 1 and conclude that maximal elements of  $\mathds{L}(A)$ and $\mathds{R}(A)$, respectively, are strongly prime, since $(\mathds{L}(A),m_{\mathds{L}})$ and $(\mathds{R}(A),m_{\mathds{R}})$ are semi-unital. Now we use the fundamental property of \donotbreakdash{$C^*$}algebras that every closed left (resp.\ right) ideal of $A$ is an intersection of maximal left (resp.\ right) ideals of $A$ (cf.\ \cite[10.2.10~Thm.\,(iii)]{Kadison2}). Hence the semi-unital quantales  $(\mathds{L}(A),m_{\mathds{L}})$ and $(\mathds{R}(A),m_{\mathds{R}})$ are strongly spatial, and Proposition~\ref{newproposition1} implies that the ideal multiplication in $\mathds{L}(A)$ and $\mathds{R}(A)$ is bisymmetric and idempotent (cf.\ Proposition~\ref{newproposition1}\,(iii)).

With regard to the previous observations, we conclude that the tensorially involutive quantale $(\mathds{L}(A)\otimes \mathds{R}(A),m_{\mathds{L}}\otimes m_{\mathds{R}},\ell)$ associated with $\mathds{M}axA$ is semi-unital, semi-integral, pre-idempotent and bisymmetric (cf.\ Remark~\ref{newremark1}). Since the one-to-one relationship between pure states of $A$ and maximal left ideals is given by the construction of left kernels (cf.\ \cite[Thm.~10.2.10\,(i) and (ii)]{Kadison2}), 
we conclude from Corollary~\ref{newcorollary3} that every pure state of $A$ can be identified with a hermitian
prime element of the tensorially involutive quantale associated with $\mathds{M}axA$.

When the \donotbreakdash{$C^*$}algebra $A$ is commutative, the quantale $\mathds{M}axA$ differs from the quantale of closed ideals of $A$. Hence, the quantale $\mathds{M}axA$ is inappropriate in order to play the role of a spectrum of a \donotbreakdash{$C^*$}algebra $A$. Therefore, we deviate from the approach taken by Mulvey and Pelletier and construct the spectrum of a \donotbreakdash{$C^*$}algebra as follows.

In the first step, we consider the quantale $\mathds{I}(A)$ of all closed two-sided ideals of $A$ provided with the ideal multiplication. Then the universal upper bound $\top=A$ is the unit  with respect to the ideal multiplication, and consequently, $\mathds{I}(A)$ is a frame, and the ideal multiplication coincides with the binary meet operation (cf.\ \cite[p.~120]{EGHK}). 

Moreover, if $\mathds{I}(A)\otimes \mathds{I}(A) \xrightarrow{\, m\,} \mathds{I}(A)$ is now the semigroup operation in the sense of $\cat{Sup}$ corresponding to the binary meet in $\mathds{I}(A)$, then $m$ is commutative and  even a quantale homomorphism 
\[(\mathds{I}(A)\otimes \mathds{I}(A), m\otimes m) \xrightarrow{\, m\,} (\mathds{I}(A),m).\]
\noindent 
In fact, if $I_1$, $I_2$, $I_3$ and $I_4$ are closed two-sided ideals of $A$, then the following relation holds for elementary tensors:
\[ m((I_1\otimes I_2)\star (I_3\otimes I_4))=m((I_1\wedge I_3) \otimes (I_2\wedge I_4))= I_1\wedge I_3\wedge I_2\wedge I_4= m(I_1\otimes I_2)\wedge m(I_3\otimes I_4).\] 

In a second step, we refer to the inclusion maps $\mathds{I}(A)\xhookrightarrow{\, q_{\mathds{L}(A)}\,}\mathds{L}(A)$ and $\mathds{I}(A)\xhookrightarrow{\, q_{\mathds{R}(A)}\,} \mathds{R}(A)$ and consider the coequalizer 
\[(\mathds{L}(A)\otimes \mathds{R}(A), m_{\mathds{L}} \otimes m_{\mathds{R}}) \xrightarrow{\,\pi\,} 
(\mathds{V},m_{\mathds{V}})\] 
\noindent of $\mathds{I}(A) \xleftleftarrows[\, j_{\mathds{L}(A)}\circ  q_{\mathds{L}(A)}\,]{\,j_{\mathds{R}(A)}\circ  q_{\mathds{R}(A)}\,} \mathds{L}(A)\otimes \mathds{R}(A)$ in the sense of $\cat{Quant}$. Then we view the range of $\pi$ as the \emph{quantale} representing the concept of \emph{non-commutative closed ideals} of $A$. Hence, the \emph{spectrum} of $A$ is given by $(\mathds{V},m_{\mathds{V}})$. 

Before we proceed, we explore some properties of the spectrum of $A$.

First, we notice that the domain of $\pi$ is precisely the tensorially involutive quantale associated with $\mathds{M}axA$. Hence, the spectrum of $A$ is a semi-unital, semi-integral, pre-idempotent and bisymmetric quantale. Moreover, since every closed two-sided ideal is self-adjoint (cf.\ \cite[p.~239 and 4.2.10~Cor.]{Kadison1}), the inclusion maps 
$\mathds{I}(A)\xhookrightarrow{\, q_{\mathds{L}(A)}\,}\mathds{L}(A)$ and $\mathds{I}(A)\xhookrightarrow{\, q_{\mathds{R}(A)}\,} \mathds{R}(A)$ satisfy the property:
\begin{equation*}\vartheta_{\mathds{L}(A)}\circ q_{\mathds{L}(A)}=q_{\mathds{R}(A)} \quad \text{and}\quad \vartheta_{\mathds{R}(A)}\circ q_{\mathds{R}(A)}=q_{\mathds{L}(A)}.
\end{equation*}
\noindent 
where $\mathds{L}(A) \xrightarrow{\, \vartheta_{\mathds{L}(A)}\,} \mathds{R}(A)$ and $\mathds{R}(A) \xrightarrow{\, \vartheta_{\mathds{R}(A)}\,} \mathds{L}(A)$ are given by
\[\vartheta_{\mathds{L}(A)}(M)=M^{\prime} \quad \text{and} \quad \vartheta_{\mathds{R}(A)}(N)=N^{\prime}, \qquad M\in \mathds{L}(A),\, N\in \mathds{R}(A).\]
Then we conclude from Theorem~\ref{newtheorem2} that the nucleus associated with $\pi$ is involutive, and so  Proposition~\ref{newproposition2} implies that there exists an involution on the spectrum of $A$ turning the coequalizer $\pi$ into an involutive quantale epimomorphism. Thus, the spectrum of $A$ is in fact an involutive quantale $(\mathds{V},m_{\mathds{V}},\ell_{\mathds{V}})$, which we will denote by $\Omega_q(A)$ (cf.\ \cite[Rem.~5.3]{GH24}).

We finish this example with a discussion concerning the commutative setting. So, if $A$ is commutative, then $\mathds{L}(A)=\mathds{R}(A)=\mathds{I}(A)$, $q_{\mathds{L}(A)}=q_{\mathds{R}(A)}=1_{\mathds{I}(A)}$, and the coequalizer of 
\[\mathds{I}(A) \xleftleftarrows[\, j_{\mathds{L}(A)}\,]{\,j_{\mathds{R}(A)}\,} \mathds{I}(A)\otimes \mathds{I}(A)\]
 is the semigroup operation $\mathds{I}(A)\otimes \mathds{I}(A)\xrightarrow{\,m\,} \mathds{I}(A)$ corresponding to the ideal multiplication in $\mathds{I}(A)$.
In fact, if $\mathds{I}(A)\otimes \mathds{I}(A)\xrightarrow{\,k\,} \mathds{D}$  is a quantale homomorphism with $k\circ j_{\mathds{L}(A)}=k\circ j_{\mathds{R}(A)}$, then every quantale homomorphism $\varphi$ making the diagram 
\stepcounter{num}
\begin{equation}\label{neweq2.10}
\begin{tikzcd}[column sep=25pt,row sep=14pt
   ]
\mathds{I}(A)\otimes \mathds{I}(A)\arrow{r}{m}\arrow[swap]{d}{k}&\mathds{I}(A)\arrow{dl}{\varphi}\\
\mathds{D}
\end{tikzcd}
\end{equation}
commutative has the form $\varphi=k\circ j_{\mathds{L}(A)}\, (=k\circ j_{\mathds{R}(A)})$. Obviously, the commutativity of (\ref{neweq2.10}) implies for all $I\in \mathds{I}(A)$:
\[k(j_{\mathds{L}(A)}(I))= \varphi(m(j_{\mathds{L}(A)}(I)))=\varphi(m(I\otimes A))=\varphi(I\ast A)=\varphi(I),\]
i.e.\ 
$\varphi$ is uniquely determined by $k$. 
On the other hand, in order to verify the existence of $\varphi$, we define $\varphi:= k\circ j_{\mathds{L}(A)}$ and observe that for elementary tensors $I\otimes J\in \mathds{I}(A)\otimes \mathds{I}(A)$ the following relation holds:
\begin{align*}k(I\otimes J)&= k((A\otimes J)\star (I\otimes A))= k(j_{\mathds{R}(A)}(J))\ast k(j_{\mathds{L}(A)}(I))=k(j_{\mathds{L}(A)}(J))\ast k(j_{\mathds{L}(A)}(I))\\
&= k(j_{\mathds{L}(A)}(J \ast I))=k(j_{\mathds{L}(A)}(I\ast J))= \varphi(I\ast J) =(\varphi\circ m)(I\otimes J)
\end{align*}
where we have used the commutativity of the ideal multiplication in $\mathds{I}(A)$.

Since closed two-sided ideals are always self-adjoint (cf.\ \cite[Cor.~4.2.10]{Kadison1}), the involution on $\mathds{I}(A)$ is always  the identity. Hence, if $A$ is \emph{commutative}, we conclude that $\Omega_q(A)$ is isomorphic to the frame $(\mathds{I}(A),m)$ provided with the identity as involution. Thus, in the commutative setting the \emph{spectrum of $A$} is equivalent to the \emph{traditional spectrum of $A$} given by the frame of all closed ideals  of $A$.
\end{example}

Since the spectrum of a unital \donotbreakdash{$C^*$}algebra $A$ is balanced and bisym\-me\-tric, the construction in Example~\ref{newexample2} has a categorical base. Indeed, $\mathds{L}(A)\otimes \mathds{R}(A)$ is the coproduct of $\mathds{L}(A)$ and $\mathds{R}(A)$ with the coprojections $j_{\mathds{L}(A)}$ and $j_{\mathds{R}(A)}$ in $\cat{BSQuant}$ (cf.\ \cite[Thm.~2.5.9]{EGHK}). Hence, the coequalizer $\mathds{L}(A)\otimes \mathds{R}(A) \xrightarrow{\,\pi\,} \Omega_q(A)$ of $\mathds{I}(A) \xleftleftarrows[\, j_{\mathds{L}(A)}\circ q_{\mathds{L}(A)}\,]{\,j_{\mathds{R}(A)}\circ q_{\mathds{R}(A)}\,} \mathds{L}(A)\otimes \mathds{R}(A)$ in $\cat{Quant}$ is also the coequalizer in $\cat{BSQuant}$, and consequently, $\Omega_q(A)$  is the \emph{pushout} of the inclusion arrows $\mathds{I}(A)\xhookrightarrow{\, q_{\mathds{L}(A)}\,} \mathds{L}(A)$ and ${\mathds{I}(A)\xhookrightarrow{\, q_{\mathds{R}(A)}\,} \mathds{R}(A)}$ in the sense of $\cat{BSQuant}$.

In the following, we show that the subquantale $\mathds{L}(\Omega_q(A))$ of all left-sided elements of $\Omega_q(A)$  is isomorphic to $\mathds{L}(A)$. For this purpose we consider the unital quantale $\mathds{M}axA$ (cf. Example~\ref{newexample2}) and consider its semi-integral regularization $\odot$ --- i.e.\ 
\[M\odot N= M\ast A \ast N,\qquad M,N\in \mathds{M}axA.\]
Then $(\mathds{M}axA,m_{\odot})$ (where $m_{\odot}$ denotes the corresponding semigroup operation in $\cat{Sup}$) is a bisym\-me\-tric and semi-unital quantale (cf.\ \cite[Prop.~2.5.3]{EGHK}). Hence, it is an object of $\cat{BSQuant}$ with the inclusion arrows $\mathds{L}(A)\xhookrightarrow{\,\iota _{\mathds{L}(A)}\,} \mathds{M}axA$ and $\mathds{R}(A)\xhookrightarrow{\, \iota _{\mathds{R}(A)}\,} \mathds{M}axA$. Since ${\mathds{L}(A)\otimes \mathds{R}(A)}$
 is the coproduct of $\mathds{L}(A)$ and $\mathds{R}(A)$ in $\cat{BSQuant}$, there exists a unique strong quantale homorphism $(\mathds{L}(A)\otimes\mathds{R}(A),m_{\mathds{L}}\otimes m_{\mathds{R}})\xrightarrow{\, h\,} (\mathds{M}axA,m_{\odot})$ satisfying the conditions $\iota _{\mathds{L}(A)}=h\circ j_{\mathds{L}(A)}$ and $\iota _{\mathds{R}(A)}=h\circ j_{\mathds{R}(A)}$. Moreover, for all elementary tensors the following relation holds:
\begin{equation*}
h(a\otimes b)=h((A\otimes b)\star (a\otimes A))= b \odot a =b \ast a, \qquad a\in \mathds{L}(A),\, b\in \mathds{R}(A).
\end{equation*}
Hence, $h\circ j_{\mathds{L}(A)}\circ q_{\mathds{L}(A)}=h\circ j_{\mathds{R}(A)}\circ q_{\mathds{R}(A)}$ follows, and $h$ factors through 
${\mathds{L}(A)\otimes \mathds{R}(A) \xrightarrow{\,\pi\,} \Omega_q(A)}$
--- i.e.\ there exists a unique strong quantale homomorphism ${\Omega_q(A)\xrightarrow{\, \widehat{h}\,} \mathds{M}axA}$ such that $\widehat{h}\circ \pi=h$ holds.
In particular, $\widehat{h}\circ \pi \circ j_{\mathds{L}(A)}=h \circ j_{\mathds{L}(A)}=\iota_{\mathds{L}(A)}$ follows, and so $\pi\circ j_{\mathds{L}(A)}$ is injective  --- i.e.\ a monomorphism in $\cat{Quant}$. Further, the image of $\pi\circ j_{\mathds{L}(A)} $ is contained in $\mathds{L}(\Omega_q(A))$. On the other hand,  we choose $a\otimes b\in \mathds{L}(A)\otimes \mathds{R}(A)$ and assume that  $\pi(a\otimes b)$ is left-sided in $\Omega_q(A)$. Then the relation
$\pi(a\otimes b)=\pi(a\otimes (A\ast b))$ follows. Since $A\ast b$ is a closed two-sided ideal of $A$, the construction of $\pi$ implies:
\begin{align*}
\pi(a\otimes b)&=\pi(a\otimes (A\ast b))=\pi((A\otimes (A\ast b))\star (a\otimes A))=\pi( A\otimes(A\ast b))\ast  \pi(a\otimes A)\\
&=\pi( (A\ast b)\otimes A)\ast \pi(a\otimes A)=\pi(((A\ast b)\ast a)\otimes A)=(\pi\circ j_{\mathds{L}(A)})(A\ast (b\ast a)).
\end{align*}
Now we use the fact that every tensor is a join of an appropriate family of elementary tensors and conclude from the previous relation that $\mathds{L}(\Omega_q(A))\subseteq (\pi\circ j_{\mathds{L}(A)})(\mathds{L}(A))$ holds. 
Hence, $\mbox{$\pi\circ j_{\mathds{L}(A)}$}$ can be viewed as an isomorphism onto $\mathds{L}(\Omega_q(A))$ in the sense of $\cat{Quant}$ and $\cat{BSQuant}$, respectively.
 In particular, $\mathds{L}(\Omega_q)$ satisfies condition \ref{A}.

More precisely, this situation can be expressed by the following commutative diagram:
\[\begin{tikzcd}[column sep=29pt,row sep=14pt
   ]
&\mathds{L}(A) \arrow[hookrightarrow]{r}{\iota_{\mathds{L}(A)}}&
\mathds{M}axA \arrow[hookleftarrow]{r}{\iota_{\mathds{R}(A)}}&\mathds{R}(A)
\\
&\mathds{L}(\Omega_q(A))\arrow[u,"\widehat{h}_{\mathds{L}}"]  \arrow[hookrightarrow]{r}{\iota_{\mathds{L}(\Omega_q(A))}}&
\Omega_q(A) \arrow[u,swap,"\widehat{h}"] \arrow[hookleftarrow]{r}{\iota_{\mathds{R}(\Omega_q(A))}}&\Omega_q(A)\arrow[u,swap,"\widehat{h}_{\mathds{R}}"] 
\\
\mathds{I}(A) \arrow[hookrightarrow]{r}{q_{\mathds{L}(A)}}&\mathds{L}(A) \arrow{r}{j_{\mathds{L}(A)}}\arrow[u,"\pi_{\mathds{L}}"]&\mathds{L}(A)\otimes \mathds{R}(A) \arrow[u,swap,"\pi"]\arrow[leftarrow]{r}{j_{\mathds{R}(A)}}&\mathds{R}(A) \arrow[u,swap,"\pi_{\mathds{R}}"]\arrow[hookleftarrow]{r}{q_{\mathds{R}(A)}}&\mathds{I}(A) 
\\
\mathds{I}(\Omega_q(A)) \arrow[hookrightarrow]{r}{q_{\mathds{L}(\Omega_q(A))}}\arrow[u,"\widehat{h}_{\mathds{I}}"]&\mathds{L}(\Omega_q(A)) \arrow{r}{j_{\mathds{L}(\Omega_q(A))}}\arrow[u,"\widehat{h}_{\mathds{L}}"]&\mathds{L}(\Omega_q(A))\otimes \mathds{R}(\Omega_q(A)) \arrow[u,swap,"{\widehat{h}_{\mathds{L}}\otimes \widehat{h}_{\mathds{R}}}"]\arrow[leftarrow]{r}{j_{\mathds{R}(\Omega_q(A))}}&\mathds{R}(\Omega_q(A)) \arrow[u,swap,"\widehat{h}_{\mathds{R}}"]\arrow[hookleftarrow]{r}{q_{\mathds{R}(\Omega_q(A))}}&\mathds{I}(\Omega_q(A)) \arrow[u,swap,"\widehat{h}_{\mathds{I}}"]
\end{tikzcd}\]
where $\pi_{\mathds{L}}(\alpha)=\pi(\alpha\otimes \top)$ for $\alpha\in \mathds{L}(\alg{Q})$, $\widehat{h}_{\mathds{L}}$ is the inverse map of $\pi_{\mathds{L}}$ --- i.e. $\widehat{h}_{\mathds{L}}(\varkappa)=\alpha$ if and only if $\varkappa=\pi(\alpha \otimes \top)$ for $\varkappa \in \mathds{L}(\Omega_q(A))$ and  $\alpha\in \mathds{L}(\alg{Q})$, $\widehat{h}_{\mathds{R}}$ is the inverse map of $\pi_{\mathds{R}}$ --- i.e. $\widehat{h}_{\mathds{R}}(\varkappa)=\beta$ if and only if $\varkappa=\pi(\top \otimes \beta)$ for $\varkappa \in \mathds{R}(\Omega_q(A))$ and  $\beta\in \mathds{R}(\alg{Q})$, and $\widehat{h}_{\mathds{I}}(\delta)=\gamma$ if and only if $\delta=\pi(\gamma\otimes \top)=\pi(\top\otimes \gamma)$ for $\delta\in \mathds{I}(\Omega_q(A))$ and $\gamma\in \mathds{I}(A)$. Obviously, $\widehat{h}_{\mathds{I}}$ is a quantale isomorphism.

Now we conclude from the commutativity of the diagram that  $\mathds{L}(A)\otimes \mathds{R}(A)\xrightarrow{\,\pi\,} \Omega_q(A)$ is  the coequalizer of $\mathds{I}(A) \xleftleftarrows[\, j_{\mathds{R}(A)}\circ q_{\mathds{R}(A)}\,]{\,j_{\mathds{L}(A)}\circ q_{\mathds{L}(A)}\,} \mathds{L}(A)\otimes \mathds{R}(A)$ if and only if $\pi\circ (\widehat{h}_{\mathds{L}}\otimes \widehat{h}_{\mathds{R}})$ is the coequalizer of $\mathds{I}(\Omega_q(A)) \xleftleftarrows[\, j_{\mathds{R}(\Omega_q(A))}\circ q_{\mathds{R}(\Omega_q(A))}\,]{\,j_{\mathds{L}(\Omega_q(A))}\circ q_{\mathds{L}(\Omega_q(A))}\,} \mathds{L}(\Omega_q(A))\otimes \mathds{R}(\Omega_q(A))$. Since $\pi\circ (\widehat{h}_{\mathds{L}}\otimes \widehat{h}_{\mathds{R}})\circ j_{\mathds{L}(\Omega_q(A))}=\iota_{\mathds{L}(\Omega_q(A))}$ and $\pi\circ (\widehat{h}_{\mathds{R}}\otimes \widehat{h}_{\mathds{R}})\circ j_{\mathds{R}(\Omega_q(A))}=\iota_{\mathds{R}(\Omega_q(A))}$, the relation $\iota_{\mathds{L}(\Omega_q(A))}\sqcup \iota_{\mathds{R}(\Omega_q(A))}=\pi\circ (\widehat{h}_{\mathds{L}}\otimes \widehat{h}_{\mathds{R}})$ follows. Now we refer to the previous equivalence and conclude that $\iota_{\mathds{L}(\Omega_q(A))}\sqcup \iota_{\mathds{R}(\Omega_q(A))}$ is the coequalizer of $\mathds{I}(\Omega_q(A)) \xleftleftarrows[\, j_{\mathds{R}(\Omega_q(A))}\circ q_{\mathds{R}(\Omega_q(A))}\,]{\,j_{\mathds{L}(\Omega_q(A))}\circ q_{\mathds{L}(\Omega_q(A))}\,} \mathds{L}(\Omega_q(A))\otimes \mathds{R}(\Omega_q(A))$. Hence the spectrum $\Omega_q(A)$ satisfies  condition \ref{C}. Since condition \ref{B} is obvious, we state the following important
\begin{fact*} The spectrum $\Omega_q(A)$ of a unital \donotbreakdash{$C^*$}algebra $A$ is always a quantic frame. In this context, the commutative quantale $\mathds{I}(\Omega_q(A))$ is considered as the {\sf commutative part} of  $\Omega_q(A)$.
\end{fact*}

\section{Quantale-enriched involutive topological spaces}
\label{sec3:}

To begin, we revisit the concept of quantale-enriched topological spaces (cf.\ \cite{GHK21}). Let $\alg{Q}=(\alg{Q},m,e)$ be a unital quantale, and consider $\alg{Q}$ as a right (left) \donotbreakdash{$\alg{Q}$}module with respect to the multiplication $m$. Obviously, a right (left) \donotbreakdash{$\alg{Q}$}submodule of $\alg{Q}$ is a subquantale of $\alg{Q}$.

Let $X$ be a set. The free right \donotbreakdash{$\alg{Q}$}module on $X$ is the complete lattice $\alg{Q}^X$, 
ordered pointwise and equipped with the right action $\boxdot$ determined by the right quantale multiplication 
--- i.e.\ $ (f\boxdot \alpha)(x)= f(x)\ast \alpha$ for each $ \alpha\in \alg{Q}$, $x\in X$ and $ f\in \alg{Q}^X$  (cf.\ \cite{JoyalTierney}). Since every set $X$ can be viewed as a discrete \donotbreakdash{$\alg{Q}$}enriched category, a
  \donotbreakdash{$\alg{Q}$}valued map $f\in \alg{Q}^X$ is also referred to as a \emph{\donotbreakdash{$\alg{Q}$}presheaf on $X$}.

A complete sublattice $\mathcal{T}\subseteq\alg{Q}^X$ is a \donotbreakdash{$\alg{Q}$}submodule of $\alg{Q}^X$ if the inclusion map is a right \donotbreakdash{$\alg{Q}$}module homomorphism  --- i.e.\ $\mathcal{T}\xhookrightarrow{\,\,\,} \alg{Q}^X$ is join-preserving and  $f\ast \alpha\in \mathcal{T}$ for all $f\in \mathcal{T}$ and $\alpha\in \alg{Q}$. A constant \donotbreakdash{$\alg{Q}$}presheaf on $X$ that takes the value $\alpha$ will always be denoted by $\underline{\alpha}$.

Further, let $\diamond$ be a quasi-magma on $\alg{Q}$ 
(cf.\ \cite[Def.~1]{GHK21}), which is a binary operation in the sense of the category of preordered sets satisfying:
\begin{equation*}
\alpha\ast (\beta\diamond\gamma)\le (\alpha\ast \beta)\diamond \gamma \quad \text{and}\quad (\beta\diamond \gamma)\ast \alpha \le \beta\diamond(\gamma\ast \alpha),\quad \alpha,\beta,\gamma\in \alg{Q}.
\end{equation*}
Obviously, the quantale  multiplication is always a quasi-magma on $\alg{Q}$. 

A \donotbreakdash{$\alg{Q}$}submodule $\mathcal{T}$ of $\alg{Q}^X$ in the sense of $\cat{Sup}$ is a \emph{\donotbreakdash{$\alg{Q}$}enriched topology} on $X$ if $\mathcal{T}$ satisfies:
 \begin{enumerate}[label=\textup{(T\arabic*)},leftmargin=28pt,topsep=3pt,itemsep=3pt]
\item \label{T1} $\underline{\top}$ is an element of $\mathcal{T}$.
\item \label{T2} If $f_1,f_2\in \mathcal{T}$ then $f_1\diamond f_2\in \mathcal{T}$, where $f_1\diamond f_2$ is defined pointwise.
\end{enumerate}
Elements of a \donotbreakdash{$\alg{Q}$}enriched topology $\mathcal{T}$ on $X$ are called \emph{open \donotbreakdash{$\alg{Q}$}presheaves}. The trivial \donotbreakdash{$\alg{Q}$}enriched topology on $X$ coincides with the \donotbreakdash{$\alg{Q}$}submodule of all constant \donotbreakdash{$\alg{Q}$}presheaves attaining all left-sided elements of $\alg{Q}$. 

A \emph{\donotbreakdash{$\alg{Q}$}enriched topological space} (\donotbreakdash{$\alg{Q}$}topological space for short) is a pair $(X,\mathcal{T})$ such that $X$ is a set and $\mathcal{T}$ is a \donotbreakdash{$\alg{Q}$}enriched topology on $X$. A \donotbreakdash{$\alg{Q}$}topological space $(Z,\mathcal{T}_Z)$ is a \emph{\donotbreakdash{$\alg{Q}$}topological subspace} of $(X,\mathcal{T})$, if there exists an injective map $Z\xrightarrow{\,\varphi\,} X$ such that $\mathcal{T}_Z=\set{f\circ \varphi\mid f\in \mathcal{T}}$
 --- i.e.\ $\mathcal{T}_Z$ is the initial structure on $Z$ with respect to $\varphi$ and $(X,\mathcal{T})$. 
Sometimes $\mathcal{T}_Z$ is referred to as the \emph{relative  \donotbreakdash{$\alg{Q}$}enriched topology} on $Z$ induced by $\mathcal{T}$ and $\varphi$.
In many instances, $Z$ is a subset of $X$ and $\varphi$ is the inclusion map $Z\xhookrightarrow{\,\,\,} X$.

Next, we recall the lower separation axioms for \donotbreakdash{$\alg{Q}$}topological spaces. A \donotbreakdash{$\alg{Q}$}topological space is classified as follows (for more details we refer to \cite{Arrieta,GHK21}): 
 \begin{enumerate}[label=\textup{--},leftmargin=12pt,topsep=3pt,itemsep=3pt,topsep=3pt,itemsep=3pt
]
\item \emph{Kolmogoroff separated} (\donotbreakdash{$\mbox{\rm T}_0$}space) For each $x,y\in X$ with $x\neq y$ there exists $f\in \mathcal{T}$ such that $f(x)\neq f(y)$
\item \emph{Fr\'echet separated} (\donotbreakdash{$\mbox{\rm T}_1$}space) For each $x,y\in X$ with $x\neq y$ there exist $f_1,f_2\in \mathcal{T}$ such that $f_1(x)\not\le f_1(y)$ and $f_2 (y) \not\le f_2(x)$.
\item \emph{Hausdorff separated} (\donotbreakdash{$\mbox{\rm T}_2$}space) For each $x,y\in X$ with $x\neq y$ there exist $f_1,f_2\in \mathcal{T}$ such that
\[f_1(x)\diamond f_2(y)\not\le \tbigvee\limits_{z\in X} f_1(z)\diamond f_2(z)\quad \text{or}\quad f_2(y)\diamond f_1(x)\not\le \tbigvee\limits_{z\in X} f_2(z)\diamond f_1(z).\]
\end{enumerate}

After these general remarks on \donotbreakdash{$\alg{Q}$}topological spaces, we now additionally consider $\alg{Q}=(\alg{Q},m,e,\ell)$ as an involutive quantale. It is natural to incorporate the involution $\ell$ into the axioms of \donotbreakdash{$\alg{Q}$}enriched topologies. To do so, we require that the quasi-magma $\diamond$ on  $\alg{Q}$ is \emph{involutive}, meaning:
\begin{equation*}
(\alpha\diamond \beta)^{\prime}= \beta^{\prime} \diamond  \alpha^{\prime},\qquad \alpha,\beta\in \alg{Q}.
\end{equation*}
A \donotbreakdash{$\alg{Q}$}enriched topology $\mathcal{T}$ on $X$ is \emph{involutive} if it satisfies the following additional condition:
 \begin{enumerate}[label=\textup{(T\arabic*)},leftmargin=28pt,topsep=3pt,itemsep=3pt,start=3]
 \item \label{T3} If $f\in \mathcal{T}$ then $f^{\prime}\in \mathcal{T}$.
\end{enumerate}
Consequently, $(X,\mathcal{T})$ is called an involutive \emph{\donotbreakdash{$\alg{Q}$}enriched topological space}.
   
The axiom \ref{T3} implies that involutive \donotbreakdash{$\alg{Q}$}enriched topologies on $X$ are \emph{\donotbreakdash{$\alg{Q}$}sub-bimodules} of the \donotbreakdash{$\alg{Q}$}bimodule $\alg{Q}^X$, where the left action is determined by the left quantale multiplication. In this context, it is important to note that the subquantale $\mathds{L}(\mathcal{T})$ of all \emph{left-sided open} \donotbreakdash{$\alg{Q}$}presheaves of an involutive \donotbreakdash{$\alg{Q}$}enriched  topology $\mathcal{T}$ is always a \donotbreakdash{$\alg{Q}$}enriched topology, but not involutive.

Let $\alg{S}$ be an involutive subquantale of $\alg{Q}$, and suppose $\alg{S}$ is also a right \donotbreakdash{$\alg{Q}$}submodule of $\alg{Q}$. An involutive \donotbreakdash{$\alg{Q}$}enriched topological space $(X,\mathcal{T})$ is called an \emph{involutive \donotbreakdash{$\alg{S}$}sober space} if $(X,\mathcal{T})$ is a \donotbreakdash{$\mbox{\textup{T}}_0$}space and every strong, involutive quantale homomorphism $\mathcal{T}\xrightarrow{\,K\,} \alg{S}$, which is also a right \donotbreakdash{$\alg{Q}$}module homo\-morphism, is induced by an element $p_0\in X$ --- this means that  $K(f)=f(p_0)$ holds for all $f\in \mathcal{T}$.

Let $\mathcal{T}$ be a \donotbreakdash{$\alg{Q}$}topology on a set $X$. A subset $\mathcal{B}\subseteq\mathcal{T}$ is called a \emph{base} of $\mathcal{T}$ if every open \donotbreakdash{$\alg{Q}$}presheaf in $\mathcal{T}$ can be expressed as a join of elements from $\mathcal{B}$.
A subset $\mathcal{S}$ of $\alg{Q}^X$ is called a \emph{subbase} of an involutive \donotbreakdash{$\alg{Q}$}enriched topology $\mathcal{T}$ on $X$ if 
\[\mathcal{T}=\tbigcap\set{\mathcal{W}\mid\mathcal{S}\subseteq \mathcal{W}\text{ and } \mathcal{W}\text{ is an involutive \donotbreakdash{$\alg{Q}$}enriched topology on $X$}}.\]
Obviously, a subbase $\mathcal{S}$ of an involutive \donotbreakdash{$\alg{Q}$}enriched  topology $\mathcal{T}$ is not necessarily a subbase of $\mathcal{T}$ in the sense of general \donotbreakdash{$\alg{Q}$}enriched topologies. 
There exist simple counterexamples in this respect.
%

A characterization of \donotbreakdash{$\alg{Q}$}interior operators (cf.\ \cite{GHK21}) in the       context of involutive \donotbreakdash{$\alg{Q}$}enriched topologies can be given as follows.

\begin{lemma}\label{newlemma3} Let $\mathcal{T}$ be a \donotbreakdash{$\alg{Q}$}enriched topology on $X$ and $\mathcal{I}$ 
be the corresponding \donotbreakdash{$\alg{Q}$}interior operator. Then  $\mathcal{T}$ is involutive if and only if $\mathcal{I}$ satisfies the following condition\textup:
\begin{enumerate}[label=\textup{(I\arabic*)},leftmargin=28pt,topsep=3pt,itemsep=3pt,start=3]
\item $\mathcal{I}(f)^{\prime}\le \mathcal{I}(f^{\prime}),\qquad f\in \alg{Q}^X$.
\end{enumerate}
\end{lemma}

\begin{proof} The sufficiency is immediate. Conversely, if $\mathcal{T}$ is an involutive \donotbreakdash{$\alg{Q}$}enriched topology on $X$, then $I(f)^{\prime}\in \mathcal{T}$ for all $f\in \alg{Q}$ and thus $I(f)^{\prime}\le I(I(f)^{\prime})\le I(f^{\prime})$.
\end{proof}

The previous condition is obviously equivalent to \begin{enumerate}[label=\textup{(I\arabic*)},leftmargin=28pt,topsep=3pt,itemsep=3pt,start=3]
\item[(I3$^\prime$)] $\mathcal{I}(f)^{\prime}= \mathcal{I}(f^{\prime}),\qquad f\in \alg{Q}^X$.
\end{enumerate}
Consequently, if the quasi-magma $\diamond$ coincides with the quantale multiplication, then Lemma~\ref{newlemma3} states that the \donotbreakdash{$\alg{Q}$}interior operator of any involutive \donotbreakdash{$\alg{Q}$}enriched topology on $X$ is always an involutive conucleus on $\alg{Q}^X$ with respect to the pointwise defined quantale multiplication.  

\section{The quantization of $\boldsymbol{2}$ and quantized topological spaces}
\label{sec:4}

Quantized topological spaces naturally emerge as a distinct subclass of quantale-enriched topological spaces through the topologization of quantales. This construction is fundamentally based on the quantization of $\boldsymbol{2}$.

  Since the concept of open \donotbreakdash{$\alg{Q}$}presheaves requires a \emph{unital} quantale $\alg{Q}$ and the quantization $\alg{Q}_{\boldsymbol{2}}$ of $\boldsymbol{2}$ is non-unital, we are interested in unital and non-integral quantales $\alg{Q}$ such that $\alg{Q}_{\boldsymbol{2}}$ can be embedded as a maximal strong subquantale of $\alg{Q}$. This approach leads to the concept of strictly quantized, unital quantales. 

A unital quantale $\widetilde{\alg{Q}_2}$ is \emph{quantized}, if $\alg{Q}_{\boldsymbol{2}}$ is a subquantale of $\widetilde{\alg{Q}_2}$ and the inclusion map is a strong quantale homomorphism. It is \emph{strictly quantized} if, additionally, the unital subquantale generated by $\alg{Q}_{\boldsymbol{2}}$ coincides with $\widetilde{\alg{Q}_2}$ (cf.\ \cite[Def.~2.1 and 2.4]{GH24}).
Thus, a strictly quantized, unital quantale contains at most $11$ elements. Besides the unit $e$ and the elements of $\alg{Q}_{\boldsymbol{2}}$, the remaining elements of $\widetilde{\alg{Q}_2}$ are denoted by $\widetilde{b}:=b\vee e$, $\widetilde{a}_{\ell}:=a_{\ell} \vee e$, $\widetilde{a}_r:=a_r \vee e$ and $\widetilde{c}:=c\vee e$. Obviously $\widetilde{\alg{Q}_2}$ is always pre-idempotent, and the cardinality of $\widetilde{\alg{Q}_2}$ depends on the order-theoretical position of the unit (cf.\ \cite{GH24}). Up to an isomorphism there exist $6$ different strictly quantized, unital quantales (cf.\ \cite[Figures~2, 4 and 5]{GH24}).

Since the unit in a quantale is always hermitian, the involution on $\alg{Q}_{\boldsymbol{2}}$ extends  uniquely  to an involution on $\widetilde{\alg{Q}_2}$. Therefore, the embedding $\alg{Q}_{\boldsymbol{2}} \xhookrightarrow{\,\,\,}\widetilde{\alg{Q}_2}$ is always involutive, and $\alg{Q}_{\boldsymbol{2}}$ becomes an involutive subquantale of $\widetilde{\alg{Q}_2}$.
   
Furthermore, recall that every unital quantale can be regarded as a right \donotbreakdash{$\alg{Q}$}module via right quantale multiplication. In this context, $\alg{Q}_{\boldsymbol{2}}$ is a right \donotbreakdash{$\widetilde{\alg{Q}_2}$}submodule of $\widetilde{\alg{Q}_2}$. Since the elements $\widetilde{a_{\ell}}$ and $\widetilde{a_r}$ are always present in any strictly quantized unital quantale and the relations
\[ \alpha \ast \widetilde{b}=\alpha\ast \widetilde{a_{\ell}}=\alpha \quad \text{and} \quad \alpha \ast a_r=\alpha\ast \top\]
hold for all $\alpha\in \alg{Q}_{\boldsymbol{2}}$, the property of a subset of $\alg{Q}_{\boldsymbol{2}}$ being a \donotbreakdash{$\widetilde{\alg{Q}_2}$}submodule of $\alg{Q}_2$ does \emph{not} depend on the specific choice of strictly quantized, unital quantale. In fact, there exist  three \donotbreakdash{$\widetilde{\alg{Q}_2}$}submodules of $\alg{Q}_{\boldsymbol{2}}$ that  contain the universal upper bound $\top$ of $\alg{Q}_{\boldsymbol{2}}$:
\[\{\bot,a_{\ell},\top\}, \quad \{\bot,a_{\ell}, c,\top\} \quad \text{and} \quad \alg{Q}_{\boldsymbol{2}}.\]
All of these \donotbreakdash{$\widetilde{\alg{Q}_2}$}submodules are necessarily subquantales of $\alg{Q}_{\boldsymbol{2}}$, but only $\alg{Q}_{\boldsymbol{2}}$ itself is involutive.
 
Since we are not concerned with the cardinality of strictly quantized unital quantales, we will choose a strictly quantized unital quantale $\widetilde{\alg{Q}_2}$ and establish the following:

\begin{standingassumption*} In what follows, $\widetilde{\alg{Q}_2}$ will always denote a fixed strictly quantized unital quantale, and the quasi-magma $\diamond$ will always coincide with the quantale multiplication $\ast$.
\end{standingassumption*}

A \emph{quantized topological space} is a \donotbreakdash{$\widetilde{\alg{Q}_2}$}topological space $(X,\mathcal{T})$ satisfying the additional property that the range of every open \donotbreakdash{$\widetilde{\alg{Q}_2}$}presheaf is contained in $\alg{Q}_2$ --- i.e.\ $\mathcal{T}$ is a \donotbreakdash{$\widetilde{\alg{Q}_2}$}submodule of $\alg{Q}_{\boldsymbol{2}}^X$ (cf. \cite[Section~5]{GH24}). 
Since both the right action on $\mathcal{T}$  and the formation of joins in $\mathcal{T}$ are defined pointwise, it follows from the previous observations that the property of being a right \donotbreakdash{$\widetilde{\alg{Q}_{\boldsymbol{2}}}$}submodule of $\alg{Q}_{\boldsymbol{2}}^X$ is independent of the specific choice of strictly quantized, unital  quantales. For this reason, $\mathcal{T}$ is also referred to a \emph{quantized topology}. A quantized topological space $(X,\mathcal{T})$ is \emph{involutive} if $\mathcal{T}$ satisfies the axiom (\mbox{T}3).

A quantized topological space $(X,\mathcal{T})$ is \emph{strongly Hausdorff separated} if for each $x,y\in X$ with $x\neq y$ there exists a left-sided  $f_1\in \mathcal{T}$ and a right-sided $f_2\in \mathcal{T}$ such that
\[f_2(x)\star f_1(y) \not\le \tbigvee\limits_{z\in X} (f_2(z)\star f_1(z)).\]
We show that the {strong Hausdorff separation} has a \emph{convergent-theoretical} characterization.
 For this purpose we recall that an isotone map $\alg{Q}^X\xrightarrow{\,\omega\,} \alg{Q}$ is called a \donotbreakdash{$\alg{Q}$}enriched filter (\donotbreakdash{$\alg{Q}$}filter {for} short) (cf.\ \cite[Def.~3]{GHK21}), if $\omega$ satisfies the following conditions:
 \begin{enumerate}[label=\textup{(F\arabic*)},leftmargin=28pt,topsep=3pt,itemsep=3pt,start=0]
\item $\omega(f)\ast \alpha\le \omega(f\ast \alpha),\qquad \alpha\in \alg{Q},\,f\in \alg{Q}^X$,
\item $\omega(\underline{\top})=\top$,
\item $\omega(f_1) \ast \omega(f_2)\le \omega(f_1\ast f_2),\qquad f_1,f_2\in \alg{Q}^X$,
\item  $\omega(f)\le \tbigvee_{x\in X} f(x),\qquad f\in \alg{Q}^X$.
\end{enumerate}
An element $x\in X$ is a \emph{left} (resp. \emph{right}) \emph{limit} of $\omega$ if $\nu_x(f)\le \omega(f)$ for every left-sided (resp. right-sided) $f\in \alg{Q}^X$, where $(\nu_x)_{x\in X}$ denotes the \donotbreakdash{$\alg{Q}$}neighborhood system corresponding to $(X,\mathcal{T})$ (cf.\ \cite{GHK21}).
%
 
\pagebreak
\begin{theorem} Let $(X,\mathcal{T})$ be a quantized topological space.
 Then the following are equivalent\textup:
 \begin{enumerate}[label=\textup{(\roman*)},leftmargin=20pt,labelwidth=10pt,itemindent=0pt,labelsep=5pt,topsep=5pt,itemsep=3pt]
\item $(X,\mathcal{T})$ is strongly Hausdorff separated.
\item For each \donotbreakdash{$\widetilde{\alg{Q}_2}$}filter $\omega$ on $X$ and  each $x,y\in X$ such that $x$ is left limit and $y$ is right limit of $\omega$ the relation $x=y$ holds --- i.e.\ if $\omega$ has left a limit and a right limit, then they coincide.
\end{enumerate}
\end{theorem}

\begin{proof} (i)$\implies$(ii): Let us assume that $x$ is a left limit of $\omega$    
and $y$ is a right limit of $\omega$. Then for every left-sided open \donotbreakdash{$\widetilde{\alg{Q}_2}$}presheaf $f_1$ and for every right-sided open \donotbreakdash{$\widetilde{\alg{Q}_2}$}presheaf $f_2$ we obtain:
\[f_2(y)\star f_1(x)=\nu_y(f_2)\star \nu_x(f_1)\le \omega(f_2)\star \omega(f_1)\le \omega(f_2\star f_1)\le \tbigvee\limits_{z\in X} (f_2(z)\star f_1(z)).\]
Hence the strong Hausdorff separation axiom implies $x=y$.\\[1mm]
(ii)$\implies$(i): Assume that $(X,\mathcal{T})$ is not strongly Hausdorff separated. Then there exists $x,y\in X$ with $x\neq y$ such that  for all left-sided open \donotbreakdash{$\widetilde{\alg{Q}_2}$}presheaves $f_1\in\mathcal{T}$ and for all right-sided open \donotbreakdash{$\widetilde{\alg{Q}_2}$}presheaves $f_2\in\mathcal{T}$, the following relation holds:
\stepcounter{num}
\begin{equation} \label{neweq4.1}
f_2(y)\star f_1(x)\le \tbigvee\limits_{z\in X} (f_2(z)\star f_1(z)).
\end{equation}
Now we construct a \donotbreakdash{$\widetilde{\alg{Q}_2}$}filter as follows:
\[\omega(h)=\tbigvee\set{g(y)\star f(x)\mid f,g \in \mathcal{T} \text{ with } \top\star f\le f,\,\, g\star \top\le g\text{ and }g\star f\le h},\qquad h\in \widetilde{\alg{Q}_2}^X.\]
Obviously, $\omega$ is isotone and satisfies axiom (F1). If $f$ is a left-sided open \donotbreakdash{$\widetilde{\alg{Q}_2}$}presheaf, then $f\star \alpha$ is also a left-sided open $\mbox{\donotbreakdash{$\widetilde{\alg{Q}_2}$}presheaf}$ for all $\alpha\in \alg{Q}$. Hence $\omega$ satisfies (F0). \\
To verify (F2), we proceed as follows. Let $h_1,h_2\in \widetilde{\alg{Q}_2}^X$ and $f_1,g_1,f_2,g_2\in \mathcal{T}$ such that $\top\star f_1\le f_1$, $g_1\star \top\le g_1$, $g_1\star f_1\le h_1$, $\top\star f_2\le f_2$, $g_2\star \top\le g_2$ and $g_2\star f_2\le h_2$. Since $(X,\mathcal{T})$ is \emph{quantized} and $\alg{Q}_{\boldsymbol{2}}$ is \emph{bisym\-me\-tric}, we have 
$f_1\star f_2,g_1\star g_2 \in \mathcal{T}$, $\top\star (f_1\star f_2)\le f_1\star f_2$, $(g_1\star g_2)\star \top\le g_1\star g_2$, $(g_1\star g_2)\star (f_1\star f_2)=(g_1\star f_1)\star (g_2\star f_2)\le h_1\star h_2$ and 
\[(g_1(y)\star f_1(x))\star (g_2(y)\star f_2(x))=(g_1\star g_2)(y)\star (f_1\star f_2)(x)\le \omega(h_1\star h_2).\]
Hence, $\omega(h_1)\star \omega(h_2)\le \omega(h_1\star h_2)$ follows. Finally, (F3) is a consequence of (\ref{neweq4.1}).\\
We show now that $x$ is left limit point of $\omega$. Let $h\in\widetilde{\alg{Q}_2}^X$ be a left-sided \donotbreakdash{$\widetilde{\alg{Q}_2}$}presheaf and $f\in \mathcal{T}$ such that $f\le h$. Then $\top\star f$ is a left-sided open \donotbreakdash{$\widetilde{\alg{Q}_2}$}presheaf, $\underline{\top}$ a right-sided open \donotbreakdash{$\widetilde{\alg{Q}_2}$}presheaf and $\underline{\top}\star(\top\star f)\le \top\star h\le h$. Hence, $\top\star f(x)=\underline{\top}(y)\star(\top\star f)(x)\le \omega(h)$. Since  $(X,\mathcal{T})$ is quantized and $\alg{Q}_{\boldsymbol{2}}$ is \emph{semi-unital}, we conclude that $f(x)\le \omega(h)$, and so
\begin{align*}\nu_x(h)&=\tbigvee\set{f(x)\mid f\in \mathcal{T}\text{ and } f\le h}\le \omega(h).
\end{align*}

Analogously, we show that $y$ is a right limit of $\omega$. Since $x\neq y$, we arrive at a contradiction of (ii). Hence, the assumption is false, and (i) holds.
\end{proof}

From the perspective of convergence theory the previous theorem underlines the importance of the strong Hausdorff separation axiom for quantized topological spaces.

\subsection{The topological characterization of strongly spatial quantales}
\label{subsec:4.1}
Consider an arbitrary quantale $\alg{Q}$ and its strong spectrum $\sigma_s(\alg{Q})$. We identify every strongly prime element $p$ of $\alg{Q}$ with its associated strong quantale homomorphism $\alg{Q} \xrightarrow{\,h_p\,} \alg{Q}_{\boldsymbol{2}}$ (cf.\ Proposition~\ref{newproposition3}). Consequently, every element $\alpha\in \alg{Q}$ induces a map $\sigma_s(\alg{Q})\xrightarrow{\, \mathds{A}_{\alpha}\,} \alg{Q}_{\boldsymbol{2}}$ defined by:
\stepcounter{num}
\begin{equation}
\label{neweq4.2} \mathds{A}_{\alpha}(p)=h_p(\alpha),\qquad p\in \sigma_s(\alg{Q}).
\end{equation}

We then consider the \donotbreakdash{$\widetilde{\alg{Q}_2}$}enriched topology generated by the subbase
 $\mathcal{S}=\set{\mathds{A}_{\alpha}\mid  \alpha \in{\alg{Q}}}$.
Since $\mathds{A}_{\alpha}\star a_{\ell}=\mathds{A}_{\alpha\ast\top}\star a_{\ell}$ for each 
$\alpha\in\alg{Q}$, we conclude from (\ref{neweq2.6}) and (\ref{neweq2.7}) that the quantized topology on $\sigma_{s(\alg{Q})}$ generated by $\mathcal{S}$ is given by
\begin{equation*}
\mathcal{T}_{\sigma_s(\alg{Q})}=\set{\mathds{A}_{\alpha} \vee (\mathds{A}_{\beta}\star a_{\ell})\mid  \alpha \in\alg{Q},\, \beta\in \mathds{R}(\alg{Q})}.
\end{equation*}

Further, we consider the tensor product $(\alg{Q}\otimes C_3, m\otimes m_{\ell})$ of the quantales $(\alg{Q},m)$ and $C_3^{\ell}$.  Then 
\[\mathcal{U}_{\alg{Q}}=\set{ (\alpha\otimes \top)\vee (\beta\otimes a)\mid \alpha\in \alg{Q},\,\beta\in \mathds{R}(\alg{Q}),\, \alpha \le \beta}\]
is a subquantale of  $(\alg{Q}\otimes C_3, m\otimes m_{\ell})$. In fact, if $(\alpha_1\otimes \top)\vee (\beta_1\otimes a),(\alpha_2\otimes \top)\vee (\beta_2\otimes a)\in\mathcal{U}_{\alg{Q}}$ then
$((\alpha_1\otimes \top)\vee (\beta_1\otimes a))\star ((\alpha_2\otimes \top)\vee (\beta_2\otimes a))=((\beta_1\ast \alpha_2)\otimes \top)\vee ((\beta_1\ast \beta_2)\otimes a)$.

\begin{proposition} {\rm(Cf.\ \cite[Prop.~5.5]{GH24})}\label{newproposition6} Let $\alg{Q}$ be a quantale and $\mathcal{T}_{\sigma_s(\alg{Q})}$ be its quantized topology on $\sigma_s(\alg{Q})$. 
Then $\alg{Q}$ is strongly spatial if and only if the quantales $\mathcal{T}_{\sigma_{s(\alg{Q})}}$ and $\mathcal{U}_{\alg{Q}}$ are isomorphic, and there exists an injective and strong quantale homomorphism
$\alg{Q}\xrightarrow{\,\,\,} \mathcal{U}_{\alg{Q}}$.
\end{proposition}


\begin{proof} {\bf Necessity}:
Since $\alg{Q}$ is strongly spatial, $\alg{Q}$ is semi-integral (cf.\ Proposition~\ref{newproposition1}).  Hence, the map $\alg{Q}\xrightarrow{\, \varphi\,} \mathcal{U}_{\alg{Q}}$ defined by
\[\varphi(\alpha)= (\alpha\otimes \top)\vee ((\alpha \ast \top)\otimes a), \qquad \alpha \in \alg{Q},\]
is an injective and strong quantale homomorphism. The isomorphism between $\mathcal{T}_{\sigma_{s(\alg{Q})}}$ and $\mathcal{U}_{\alg{Q}}$ follows from \cite[Prop.~5.5 and Rem.~5.6]{GH24}.\\[1mm]
{\bf Sufficiency}
Since every injective and join-preserving map reflects the respective 
orders, we conclude that $\alg{Q}$ is isomorphic to a subquantale of $\mathcal{T}_{\sigma_{s(\alg{Q})}}$ containing the universal upper bound of $\mathcal{T}_{\sigma_{s(\alg{Q})}}$. Since every quantized topology is a strongly spatial quantale, the result follows.
\end{proof}

 In the following, we assume that $\alg{Q}$ is always an \emph{involutive} quantale.
 In this context, we restrict our interest to the hermitian spectrum $X=\sigma_h(\alg{Q})$ and identify every hermitian strong prime element $p$ with its associated involutive strong quantale homomorphism $\alg{Q}\xrightarrow{\, h_p\,} \alg{Q}_{\boldsymbol{2}}$ (cf.\ Proposition~\ref{newproposition3}). 
 
Since $(\mathds{A}_{\alpha})^{\prime}(p)=\mathds{A}_{\alpha^{\prime}}(p)$ for all $p\in  X$ (where $\mbox{}^{\prime}$ denotes the respective involutions), we refer again to (\ref{neweq2.6}) and (\ref{neweq2.7}) and conclude that the involutive \donotbreakdash{$\widetilde{\alg{Q}_2}$}enriched quantized topology on $X$ generated by $\mathcal{S}$ is given by:
\stepcounter{num}
\begin{equation}\label{neweq4.3}
\mathcal{T}_X=\set{\mathds{A}_{\alpha_1}\vee (\mathds{A}_{\alpha_2}\star a_{\ell}) \vee (a_r\star \mathds{A}_{\alpha_3})\vee (a_r\star \mathds{A}_{\alpha_4}\star a_{\ell})\mid\alpha_i\in \alg{Q},\, i=1,...,4}.
\end{equation}
It is easy to check that $(X,\mathcal{T}_X)$ is a \donotbreakdash{$\mbox{\textup{T}}_0$}space. Moreover, the strong quantale homomorphism $\alg{Q}\xrightarrow{\ \Phi\,} \mathcal{T}_X$ determined by $\Phi(\alpha)=\mathds{A}_{\alpha}$ for all $\alpha\in\alg{Q}$ is an involutive quantale homomorphism.

\begin{proposition}\label{newproposition7} Let $\alg{Q}$ be an involutive quantale and $X=\sigma_h(\alg{Q})$ be its hermitian spectrum. Then the involutive quantized topological space $(X,\mathcal{T}_X)$ given by \textup(\ref{neweq4.3}\textup) is 
an involutive \donotbreakdash{$\alg{Q}_{\boldsymbol{2}}$}sober space.
\end{proposition}

\begin{proof} Let $\mathcal{T}_X\xrightarrow{\, K\,} \alg{Q}_{\boldsymbol{2}}$ be an involutive quantale homomorphism, which is also a right \donotbreakdash{$\widetilde{\alg{Q}_2}$}module homomorphism. Then 
the composition $\alg{Q}\xrightarrow{\ K\circ \Phi\,} \alg{Q}_{\boldsymbol{2}}$ is again an involutive strong quantale homomorphism, and it is induced by the hermitian strong prime element $p_0=(K\circ \Phi)^{\vdash}(c)$ --- i.e.\ $K\circ \Phi=h_{p_0}$. In particular
\stepcounter{num}
\begin{equation}\label{neweq4.4} 
K(\mathds{A}_{\alpha})= (K\circ \Phi)(\alpha)= h_{p_0}(\alpha)=
\mathds{A}_{\alpha}(p_0) , \qquad \alpha\in \alg{Q}.
\end{equation}
Since $K$ is involutive and a right \donotbreakdash{$\widetilde{\alg{Q}_2}$}module
 homomorphism, we conclude from  (\ref{neweq4.4}) that for all $\beta\in \alg{Q}$ the following relations hold:
\begin{align*}&K(\mathds{A}_{\alpha}\star a_{\ell})= K(\mathds{A}_{\alpha})\star a_{\ell}=\mathds{A}_{\alpha}(p_0)\star a_{\ell}=(\mathds{A}_{\alpha}\star a_{\ell})(p_0),\\
&K(a_r\star  \mathds{A}_{\alpha})= (K((\mathds{A}_{\alpha})^{\prime}\star a_{\ell}))^{\prime}= (K(\mathds{A}_{\alpha})^{\prime}\star a_{\ell})^{\prime}=(\mathds{A}_{\alpha}(p_0)^{\prime}\star a_{\ell})^{\prime}= a_r\star \mathds{A}_{\alpha}(p_0)=(a_r\star\mathds{A}_{\alpha})(p_0),\\
&K(a_r\star  \mathds{A}_{\alpha}\star a_{\ell})=K(a_r\star  \mathds{A}_{\alpha})\star a_{\ell}=a_r\star \mathds{A}_{\alpha}(p_0)\star a_{\ell}=(a_r\star  \mathds{A}_{\alpha}\star a_{\ell})(p_0).
\end{align*}
Hence, $K(g)=g(p_0)$ follows  for all $g\in \mathcal{T}_X$ (cf.\ (\ref{neweq4.3})). 
\end{proof}

In the following we are interested in involutive quantales  such that the associated quantized topological spaces have Fr\'echet separated or strongly Hausdorff separated subspaces. As we will see, quantic frames will form an interesting class for this purpose.

\section{Topologization of quantic frames}
\label{sec:5}

Let  $\alg{Q}$ be a quantic frame, $X=\sigma_h(\alg{Q})$, and $(X,\mathcal{T}_X)$ be the involutive quantized topological space associated with $\alg{Q}$ in the sense of (\ref{neweq4.3}). Then the bijective map $\sigma(\mathds{L}(\alg{Q})) \xrightarrow{\,\varphi\,}\sigma_h(\alg{Q})$ given by $\varphi(q)=q\vee q^{\prime}$ for each $q\in \sigma(\mathds{L}(\alg{Q}))$ (cf. Corollary~\ref{newcorollary4}) induces an involutive quantized topology $\mathcal{T}_Y$ on $Y=\sigma(\mathds{L}(\alg{Q}))$  as follows:
\[\mathcal{T}_Y=\set{f\circ \varphi\mid f\in \mathcal{T}_X}.\]
Since in various cases the quantale $\mathds{L}(\alg{Q})$ is easier to understand than the corresponding quantic frame $\alg{Q}$ (cf.\ Example~\ref{newexample2}), we now consider the involutive quantalized topological space $(Y,\mathcal{T}_Y)$ as the topological space associated with  $\alg{Q}$. 

The aim of the following considerations is to construct a subbase of $\mathcal{T}_Y$. 

Referring to Proposition~\ref{newproposition3}, 
we recall that each prime element $p$ of $\mathds{L}(\alg{Q})$ induces an involutive strong quantale homomorphism $\alg{Q} \xrightarrow{\, h_{\varphi(p)}\,} \alg{Q}_{\boldsymbol{2}}$, 
and a pair of strong quantale homomorphisms $\mathds{L}(\alg{Q}) \xrightarrow{\, h_{p}\,} \alg{Q}_{\boldsymbol{2}}$ and $\mathds{R}(\alg{Q}) \xrightarrow{\, h_{p^{\prime}}\,} \alg{Q}_{\boldsymbol{2}}$ by (cf.\ (\ref{neweq2.8})):
\[h_{p}(\alpha)=\begin{cases} \bot, & \alpha\ast \top\le p,\\ a_{\ell}, &  \alpha\ast \top \not\le p, \,\alpha\le p,\\
\top,  & \alpha\not\le p,
\end{cases}\quad
h_{p^{\prime}}(\beta)=\begin{cases} \bot, & \top\ast \beta\le p^{\prime},\\ a_r, &  \top\ast\beta \not\le p^{\prime}, \,\beta\le p^{\prime},\\
\top, & \beta\not\le p^{\prime}, \end{cases}  \quad \alpha \in \mathds{L}(\alg{Q}),\,\beta\in \mathds{R}(\alg{Q}).\]
Now we refer to Remark~\ref{newremark4}, Proposition~\ref{newproposition4}\,(c) and (d), and Corollary~\ref{newcorollary4} and observe that $\varphi(p)=p\vee p^{\prime}= \pi((p\otimes \top)\vee (\top\otimes p^{\prime}))\in \sigma_h(\alg{Q})$ and $\pi^{\vdash}(\varphi(p))=(p\otimes \top)\vee(\top\otimes  p^{\prime})=r\in \sigma_h(\mathds{L}(\alg{Q})\otimes\mathds{R}(\alg{Q}))$.
Hence, we consider the composition $H_r=h_{\varphi(p)}\circ \pi$ and conclude that it is the coproduct of $h_p$ and $h_{p^{\prime}}$ --- i.e.\ $H_r=h_p\sqcup h_{p^{\prime}}$. Consequently, since $\alg{Q}$ is a quantic frame, the following diagram is commutative:
\[\begin{tikzcd}[column sep=25pt,row sep=10pt]
&\alg{Q}_2&\\
&\alg{Q}\arrow{u}{h_{\varphi(p)}}&\\
\mathds{L}(\alg{Q})\arrow[r,"j_{\mathds{L}(\alg{Q})}" near end]\arrow[ru,"\iota_{\mathds{L}(\alg{Q})}" near end]\arrow[ruu, bend left, "h_p"]&\mathds{L}(\alg{Q})\otimes \mathds{R}(\alg{Q})\arrow{u}{\pi}&\mathds{R}(\alg{Q})\arrow[lu,swap,"\iota_{\mathds{R}(\alg{Q})}" near end]\arrow[l,swap,"j_{\mathds{R}(\alg{Q})}" near end]\arrow[luu, swap,bend right, "h_{p'}"]
\end{tikzcd}\]
and for each $\alpha\in\mathds{L}(\alg{Q})$ and $\beta\in\mathds{R}(\alg{Q})$ the relation
\[h_{\varphi(p)}(\pi(\alpha\otimes\beta))=h_{\varphi(p)}(\pi(\top\otimes\beta))\star h_{\varphi(p)}(\pi(\alpha\otimes \top))= h_{p^{\prime}}(\beta)\star h_{p}(\alpha)\]
holds.
\pagebreak
 Since $\mathds{L}(\alg{Q})\otimes \mathds{R}(\alg{Q}) \xrightarrow{\,\pi\,} \alg{Q}$ (cf.\ Remark~\ref{newremark4}) is surjective and every tensor in $\mathds{L}(\alg{Q}) \otimes \mathds{R}(\alg{Q})$ is a join of elementary tensors,
 we choose $\alpha\in\mathds{L}(\alg{Q})$, $\beta\in\mathds{R}(\alg{Q})$ and $\mathds{A}_{\pi(\alpha\otimes \beta)}$ as on open \donotbreakdash{$\alg{Q}$}presheaf on $X$. Then we refer to (\ref{neweq4.2}) and define:
\stepcounter{num}
\begin{equation}\label{neweq5.1}
\mathds{B}_{\alpha\otimes \beta}(p):=\mathds{A}_{\pi(\alpha\otimes \beta)}(\varphi(p))=h_{\varphi(p)}(\pi(\alpha\otimes \beta))=h_{p^{\prime}}(\beta)\star h_{p}(\alpha),\qquad  p\in\sigma(\mathds{L}(\alg{Q})).
\end{equation}

\begin{remark}\label{remark6} Referring to the formulations of $h_{p}(\alpha)$ and 
$h_{p^{\prime}}(\beta)$, it is not difficult to check that the following holds:
\[\mathds{B}_{\alpha\otimes \beta}(p)=\begin{cases} \bot, & \beta^{\prime}\ast \top\le p\quad \text{or}\quad \alpha\ast\top \le p,\\ 
b, &\beta^{\prime}\ast \top\not\le p,\, \alpha\ast\top \not\le p,\,\alpha\ast\beta^{\prime}\le p, \,\beta^{\prime}\ast\alpha\le p,\\
a_{\ell}, &\beta^{\prime}\ast \top\not\le p,\, \alpha\ast\top \not\le p,\,\alpha\ast\beta^{\prime} \not\le p, \,\beta^{\prime}\ast\alpha\le p,\\
a_{r}, &\beta^{\prime}\ast \top\not\le p,\, \alpha\ast\top \not\le p,\, \beta^{\prime}\ast\alpha \not\le p, \,\alpha\ast\beta^{\prime}\le p,\\
\top,  & \alpha\ast\beta^{\prime}\not\le p,\,\beta^{\prime}\ast\alpha\not\le p,
\end{cases}\]
where we have used frequently the primality of $p$. 
%
%
%
\end{remark}

After these simplifications we note that the involutive quantized topology $\mathcal{T}_Y$ on $Y=\sigma(\mathds{L}(\alg{Q}))$ has the following base (cf.\ (\ref{neweq4.3})):
\stepcounter{num}
\begin{equation} \label{neweq5.2}
\mathcal{B}=\set{\mathds{B}_{\alpha_1\otimes \beta_1}\vee(\mathds{B}_{\alpha_2\otimes \beta_2}\star a_{\ell}) \vee  (a_r \star \mathds{B}_{\alpha_3\otimes \beta_3})\vee (a_r\star \mathds{B}_{\alpha_4\otimes \beta_4}\star a_{\ell})\mid \alpha_i\otimes\beta_i\in \mathds{L}(\alg{Q})\otimes \mathds{R}(\alg{Q}), \,i=1,..,4}.
\end{equation}



\begin{proposition}\label{newproposition8} Let $\alg{Q}$ be a quantic frame such that the subquantale $\mathds{L}(\alg{Q})$ is spatial, and let $(Y, \mathcal{T}_Y)$ be the associated  involutive quantized topological space determined by \textup(\ref{neweq5.2}\textup). 
 \begin{enumerate}[label=\textup{(\alph*)},leftmargin=20pt,labelwidth=10pt,itemindent=0pt,labelsep=5pt,topsep=5pt,itemsep=3pt]
\item If $\mathds{B}_{\alpha_1\otimes\beta_1}$ and $\mathds{B}_{\alpha_2\otimes \beta_2}$ are open \donotbreakdash{$\widetilde{\alg{Q}_2}$}presheaves of $\mathcal{T}_Y$ with $\alpha_1,\alpha_2\in \mathds{L}(\alg{Q})$ and $\beta_1,\beta_2\in \mathds{R}(\alg{Q})$, then
$\mathds{B}_{\alpha_1\otimes\beta_1}=\mathds{B}_{\alpha_2\otimes \beta_2}$ if and only if $\beta_1^{\prime}\ast \alpha_1= \beta_2^{\prime}\ast \alpha_2$ and $\alpha_1\ast\beta_1^{\prime}= \alpha_2\ast\beta_2^{\prime}$.
\item For every left-sided open \donotbreakdash{$\widetilde{\alg{Q}_2}$}presheaf
 $f\in\mathcal{T}_Y$, there exists a unique pair $(\alpha,\gamma)\in{\mathds{L}(\alg{Q})\times \mathds{I}(\alg{Q})}$ satisfying the following property\textup:
\stepcounter{num}
\begin{equation}\label{neweq5.3}
\alpha\le \gamma \quad \text{and}\quad f=\mathds{B}_{\alpha\otimes \top}\vee\bigl(\mathds{B}_{\gamma\otimes \top}\star a_{\ell}\bigr).
\end{equation}
\item If $\alg{Q}$ is a factor, then $\mathds{L}(\mathcal{T}_Y)\setminus\{\underline{a_{\ell}}\}$ is isomorphic to $\mathds{L}(\alg{Q})$. 
\end{enumerate}
\end{proposition} 

\begin{proof} (a)  Let us assume $\mathds{B}_{\alpha_1\otimes\beta_1}=\mathds{B}_{\alpha_2\otimes \beta_2}$. Then Remark~\ref{remark6} implies that for each prime element $p$ of $\mathds{L}({\alg{Q}})$ we have  
\begin{align*}
 \beta_1^{\prime}\ast \alpha_1\le p&\iff \mathds{B}_{\alpha_2\otimes \beta_2}(p)=\mathds{B}_{\alpha_1\otimes\beta_1}(p)\le a_{\ell}\iff \beta_2^{\prime}\ast \alpha_2\le p
\end{align*}
Since $\mathds{L}(\alg{Q})$ is spatial, it follows that $\beta_1^{\prime}\ast \alpha_1= \beta_2^{\prime}\ast \alpha_2$. The equality $\alpha_1\ast\beta_1^{\prime}= \alpha_2\ast\beta_2^{\prime}$ can be proved in a similar way.
Conversely, if $\beta_1^{\prime}\ast \alpha_1= \beta_2^{\prime}\ast \alpha_2$ and $\alpha_1\ast\beta_1^{\prime}= \alpha_2\ast\beta_2^{\prime}$, then Remark~\ref{remark6} implies that $\mathds{B}_{\alpha_1\otimes\beta_1}(p)=\mathds{B}_{\alpha_2\otimes \beta_2}(p)$ for each $p\in\sigma(\mathds{L}(\alg{Q})$.
\\[1mm]
(b) Since every two-sided element of $\alg{Q}$ is hermitian (cf.\ \ref{B}), and in particular so is the universal upper bound in $\alg{Q}_{\boldsymbol{2}}$, an application of (\ref{neweq5.1}) yields the following relation for all $p\in \sigma(\mathds{L}(\alg{Q}))$: 
\begin{align*}\top \star \mathds{B}_{\alpha\otimes \beta}(p)&=h_{p^{\prime}}(\top\ast \beta)\star h_{p}(\alpha)= h_{p^{\prime}}((\top\ast \beta)^{\prime})\star  h_{p}(\alpha)\\ 
& = h_{p}(\top\ast \beta)\star  h_{p}(\alpha)=
h_{p}(\top\ast \beta\ast\alpha)= \mathds{B}_{(\top\ast \beta\ast \alpha)\otimes\top}(p).
\end{align*}
Based on this observation, we refer to (\ref{neweq5.2}) and observe that every left-sided open \donotbreakdash{$\alg{Q}$}presheaf $f$ on $\sigma(\mathds{L}(\alg{Q}))$ has the form
$f= \mathds{B}_{\alpha_1\otimes \top}\vee (\mathds{B}_{\alpha_2\otimes \top}\star a_{\ell})$. Since $\mathds{B}_{\alpha_1\otimes \top}\star a_{\ell}\le \mathds{B}_{\alpha_1\otimes \top}$, we obtain:
\[f=\mathds{B}_{\alpha_1\otimes \top}\vee (\mathds{B}_{(\alpha_1\vee\alpha_2)\otimes \top}\star a_{\ell})=\mathds{B}_{\alpha_1\otimes \top}\vee (\mathds{B}_{(\alpha_1\vee\alpha_2)\otimes \top}\star\top \star a_{\ell})=\mathds{B}_{\alpha_1\otimes \top}\vee (\mathds{B}_{((\alpha_1\vee\alpha_2)\ast \top)\otimes \top}\star a_{\ell}).\]
Taking $\alpha:=\alpha_1$ and $\gamma:=(\alpha_1\vee\alpha_2)\ast\top$ the existence is verified.\\
Now we consider two pairs $(\alpha,\gamma),\, (\overline{\alpha},\overline{\gamma})\in \mathds{L}(\alg{Q})\times \mathds{I}(\alg{Q})$
satisfying (\ref{neweq5.3}). Since $\alpha\le \gamma$ and $\overline{\alpha}\le \overline{\gamma}$, Remark~\ref{remark6} implies that for each prime element $p$ of $\mathds{L}({\alg{Q}})$ the following equivalences hold:
\begin{align*}
\alpha\not\le p&\iff(\mathds{B}_{\alpha\otimes \top} \vee (\mathds{B}_{\gamma\otimes \top}\ast a_{\ell}))(p)= (\mathds{B}_{\overline{\alpha}\otimes \top} \vee (\mathds{B}_{\overline{\gamma}\otimes \top}\ast a_{\ell}))(p)=\top\iff \overline{\alpha}\not\le p\quad\text{and}\\
 \gamma\le p&\iff (\mathds{B}_{\alpha\otimes \top} \vee (\mathds{B}_{\gamma\otimes \top}\ast a_{\ell}))(p) =(\mathds{B}_{\overline{\alpha}\otimes \top} \vee (\mathds{B}_{\overline{\gamma}\otimes \top}\ast a_{\ell}))(p)=\bot\iff \overline{\gamma}\le p
\end{align*}
hold. Since $\mathds{L}(\alg{Q})$ is spatial, $\alpha=\overline{\alpha}$ and $\gamma=\overline{\gamma}$ follow, and the uniqueness is verified.\\[1mm]
(c) Since $\mathds{L}(\alg{Q})$ is spatial and a factor, then for all $\alpha\in \mathds{L}(\alg{Q})\setminus\{\bot\}$ we conclude that the property $\mathds{I}(\alg{Q})=\set{\bot,\top}$ is equivalent to $\underline{a_{\ell}}\le \mathds{B}_{\alpha\otimes \top}$. Hence the assertion follows from (b).
\end{proof}

\begin{addition*} Since quantic frames are involutive quantales, the respective assertions of Theorem~\ref{newproposition8}\,(b) and (c) hold also  for right-sided open \donotbreakdash{$\alg{Q}$}presheaves and for $\mathds{R}(\alg{Q})$ .
\end{addition*}

To describe strongly Hausdorff separated, involutive quantized topological subspaces of $\sigma(\mathds{L}(\alg{Q}))$, we need some more terminology.

Let $\alg{Q}$ be a quantic frame and $\mathds{ML}(\alg{Q})$ be the set of all maximal left-sided  elements of $\alg{Q}$. A nonempty subset $C$ of $\mathds{ML}(\alg{Q})$ is called \emph{reduced} if for each $\alpha\in C$ the relation $\tbigwedge (C\setminus\{\alpha\}) \not\le \alpha$ holds, where the infimum is calculated in $\mathds{L}(\alg{Q})$.

\begin{proposition}\label{newproposition9} Let $\alg{Q}$ be a quantic frame, $Y=\sigma(\mathds{L}(\alg{Q}))$, and $(Y, \mathcal{T}_Y)$ be the involutive quantized topological space associated with $\alg{Q}$ \textup(cf. \textup(\ref{neweq5.2}\textup)\textup). If $C\subseteq \mathds{ML}(\alg{Q})$ is reduced, then the relative involutive quantized topological subspace $(C,\mathcal{T}_C)$  is strongly Hausdorff separated.
\end{proposition}

\begin{proof} Let us choose $\alpha_1,\alpha_2\in C$ with $\alpha_1\neq \alpha_2$. Then we introduce the following elements $\beta_1,\beta_2\in \mathds{L}(\alg{Q})$:
\[\beta_1=\tbigwedge (C\setminus\{\alpha_1\}) \quad \text{and}\quad \beta_2=\tbigwedge (C\setminus\{\alpha_2\}).\]
Then $f_1=\mathds{B}_{\beta_1\otimes\top}\in \mathcal{T}_C$ is left-sided and $f_2=\mathds{B}_{\top\otimes \beta_2^{\prime}}\in \mathcal{T}_C$ is right-sided with $f_2(\alpha_2)\star f_1(\alpha_1)=
\top$, and
\begin{align*}
\tbigvee\limits_{\alpha\in C}f_2(\alpha)\star f_1(\alpha)&=\bigl(\tbigvee\limits_{\alpha\in C\setminus\{\alpha_1,\alpha_2\}}f_2(\alpha)\star f_1(\alpha)\bigr)\vee\bigl(f_2(\alpha_1)\star f_1(\alpha_1)\bigr)\vee\bigl(f_2(\alpha_2)\star f_1(\alpha_2)\bigr)\\
&\le(a_r\star a_{\ell})\vee (\top\star a_{\ell})\vee (a_r\star\top)= b\vee a_{\ell}\vee a_r=c
\end{align*}
Hence, $(C,\mathcal{T}_C)$ is strongly Hausdorff separated.
\end{proof}

We finish this section with two examples of  involutive quantized topological spaces induced by quantic frames.
\begin{example}\label{exam5} Recall that $\alg{Q}_{\boldsymbol{2}}$ is a quantic frame (cf. Remark~\ref{newremark4}\,(b)) with
 $\mathds{L}(\alg{Q}_{\boldsymbol{2}})=\set{\bot,a_{\ell},\top}$ and $\sigma(\mathds{L}(\alg{Q}_{\boldsymbol{2}}))=\set{\bot,a_{\ell}}$. Thus $\mathds{L}(\alg{Q}_{\boldsymbol{2}})$ is spatial. According to (\ref{neweq2.6}) and (\ref{neweq2.7}), the involutive quantized topology $\mathcal{T}$ on $\sigma(\mathds{L}(\alg{Q}_{\boldsymbol{2}}))$ associated with $\alg{Q}_{\boldsymbol{2}}$ in the sense of (\ref{neweq5.2}) has a base $\mathcal{B}$ consisting of $9$ elements:
\[\mathcal{B}=\set{\mathds{B}_{a_{\ell}\otimes \top},\,\mathds{B}_{\top\otimes a_r},\,\mathds{B}_{a_{\ell}\otimes a_r},\, a_r\star \mathds{B}_{a_{\ell}\otimes \top},\,\mathds{B}_{\top\otimes a_r}\star  a_{\ell},\, \underline{b},\,\underline{a_{\ell}},\,\underline{a_r},\,\underline{\top}}.\]
Then $\bigl(\sigma(\mathds{L}(\alg{Q}_{\boldsymbol{2}})),\mathcal{T}\bigr)$ is an involutive \donotbreakdash{$\alg{Q}_{\boldsymbol{2}}$}sober space (cf.\ Proposition~\ref{newproposition7}). However, it is easy to check that $f(a_{\ell})\le f(\bot)$ for all $f\in \mathcal{T}$. Hence 
$\bigl(\sigma(\mathds{L}(\alg{Q}_{\boldsymbol{2}})),\mathcal{T}\bigr)$ is not Fr\'echet separated. 

Moreover, $\mathcal{T}$ is an involutive quantale, but \emph{not} a quantic frame. Indeed, let us assume that $\mathcal{T}$ is a quantic frame. Since $\mathds{L}(\alg{Q}_{\boldsymbol{2}})$ is spatial, we conclude from Proposition~\ref{newproposition8}\,(c) that the subquantale of all left-sided (resp.\ right-sided) open \donotbreakdash{$\alg{Q}$}presheaves have the following form:
\[\mathds{L}(\mathcal{T})=\bigset{\underline{\bot},\underline{a_{\ell}}, \mathds{B}_{a_{\ell}\otimes \top},\underline{\top}} \quad \text{and}\quad \mathds{R}(\mathcal{T})=\bigset{\underline{\bot},\underline{a_r}, \mathds{B}_{\top\otimes a_r},\underline{\top}}.\]
Hence $\mathds{I}(\mathcal{T})=\set{\underline{\bot},\overline{\top}}$, and we conclude from Remark~\ref{newremark4} that $\mathds{L}(\mathcal{T})\otimes \mathds{R}(\mathcal{T}) \xrightarrow{\,\pi\,} \mathcal{T}$ is an isomorphism  making the diagram 
\[
\begin{tikzcd}[column sep=25pt,row sep=15pt]
&\mathcal{T}&\\
\mathds{L}(\mathcal{T})\arrow[hookrightarrow]{ru}{j_{\mathds{L}}}\arrow[r,"j_{\mathds{L}(\mathcal{T})}" near end]&\mathds{L}(\mathcal{T})\otimes \mathds{R}(\mathcal{T})\arrow{u}{\pi}&\mathds{R}(\mathcal{T})\arrow[hookrightarrow,swap]{lu}{j_{\mathds{R}}}\arrow[swap,l,"j_{\mathds{R}(\mathcal{T})}" near end]
\end{tikzcd}
\]
commutative. In particular $\pi$ satisfies the following property for all $f\in \mathds{L}(\mathcal{T})$ and  $g\in \mathds{R}(\mathcal{T})$:
\begin{equation*} 
\pi(f\otimes g)= \pi(\underline{\top}\otimes g)\star \pi(f\otimes \underline{\top})= j_{\mathds{R}}(g)\star  j_{\mathds{L}}(f)= g\star  f.
\end{equation*}
Now we observe that $(\underline{a_{\ell}}\otimes \underline{\top})\vee (\mathds{B}_{a_{\ell}\otimes \top}\otimes \mathds{B}_{\top\otimes a_r})\neq \mathds{B}_{a_{\ell}\otimes \top}\otimes \underline{\top}$ and
\[\pi\bigl( (\underline{a_{\ell}}\otimes \underline{\top})\vee (\mathds{B}_{a_{\ell}\otimes \top}\otimes \mathds{B}_{\top\otimes a_r})\bigr)=\underline{a_{\ell}} \vee (\mathds{B}_{\top\otimes a_r}\star  \mathds{B}_{a_{\ell}\otimes \top})= \mathds{B}_{a_{\ell}\otimes \top}= \pi(\mathds{B}_{a_{\ell}\otimes \top}\otimes \underline{\top}).\]
Consequently, $\pi$ is not injective, and so the assumption is false.
\end{example}

\begin{example} Let $\mathcal{H}$ be an infinite dimensional Hilbert space and $B(\mathcal{H})$ be the \donotbreakdash{$C^*$}algebra of all bounded linear operators $\mathcal{H}\xrightarrow{\,\,\,}\mathcal{H}$.  Then the spectrum $\Omega_q(B(\mathcal{H}))$ of non-commutative closed  ideals of $B(\mathcal{H})$ (cf.\ Example~\ref{newexample2}) is a quantic frame and $\mathds{L}(\Omega_q(B(\mathcal{H}))\cong \mathds{L}(B(\mathcal{H}))$ is spatial (cf.\ Example~\ref{newexample1}). Now let $(\sigma(\mathds{L}(B(\mathcal{H}))), \mathcal{T})$ be the involutive \donotbreakdash{$\alg{Q}_{\boldsymbol{2}}$}sober space induced by $\Omega_q(B(\mathcal{H}))$ in the sense of (\ref{neweq5.2}). Since the ideal multiplication in $\mathds{L}(B(\mathcal{H}))$ does not have zero divisors, the zero ideal in $\mathds{L}(B(\mathcal{H}))$ is prime. Hence $(\sigma(\mathds{L}(B(\mathcal{H})), \mathcal{T})$ is not Fr\'echet separated. Therefore we are interested in involutive quantized  topological subspaces.

The first choice is the space $\mathcal{P}$ of all pure states of $B(\mathcal{H})$. For this purpose we recall that every pure state $\varrho$ of $B(\mathcal{H})$ can be identified with a maximal left ideal $I_{\varrho}$ given by its left kernel
\[I_{\varrho}=\set{f\in B(\mathcal{H})\mid \varrho(f^*\cdot f)=0}.\]
Hence we denote the set of maximal left ideal of $B(\mathcal{H})$ also by $\mathcal{P}$. Since maximal left ideals are incomparable, the involutive
quantized topological subspace of all pure states of  $(\sigma(\mathds{L}(B(\mathcal{H}))), \mathcal{T})$ is Fr\'echet separated, but it is an open question whether the subspace $\mathcal{P}$ is Hausdorff separated. To obtain non-trivial strongly Hausdorff separated,   involutive quantized topological subspaces we  can follow Proposition~\ref{newproposition9} and consider reduced nonempty subsets of maximal left ideals. 

A simple example of a reduced nonempty subset of maximal left ideals can be given by a family 
$\set{\omega_{x_i}\mid i\in I}$
 of vector states (cf.\ Example~\ref{newexample1}) such that $\set{x_i\mid i\in I}$ is linearly independent. To confirm this fact, we fix a linearly independent family $\set{x_i\mid i\in I}$   of unit vectors $x_i\in \mathcal{H}$ and $i_0\in I$. Then we consider the closed linear subspace $V$ of $\mathcal{H}$ generated by $\set{x_i\mid  i\in I}$ and the closed linear subspace $W_{i_0}$ of $\mathcal{H}$ generated by $\set{x_i\mid i\in I\text{ and } i\neq i_0}$. Now we choose a unit vector $e_{i_0}\in V$ being orthogonal to $W_{i_0}$ and consider the following bounded, linear operator $\mathcal{H}\xrightarrow{\,f_{i_0}\,} \mathcal{H}$ determined by:
\[f_{i_0}(x)=\langle x,e_{i_0} \rangle e_{i_0},\qquad x\in \mathcal{H}.\]
It is easily seen that the following relations hold $f_{i_0}(x_{i_0})=\langle x_{i_0},e_{i_0}\rangle e_{i_0}\neq \vec{0}$ and $f_{i_0}(x_i)=\vec{0}$ for all $i\neq i_0$. Since $I_{x_i}=\set{f\in B(\mathcal{H})\mid f(x_i)=\vec{0}}$ (cf.\ Example~\ref{newexample1}), we conclude:
\[ f_{i_0}\not\in I_{x_{i_0}} \quad \text{and}\quad f_{i_0}\in \tbigcap\limits_{i\neq i_0} I_{x_i}.\]
Hence $\tbigcap_{i\neq i_0}I_{x_i}\not\subseteq I_{x_{i_0}}$ follows --- i.e.\ $\set{I_{x_i}\mid i\in I}$ is a reduced subset of maximal left ideals. The cardinality of $\set{I_{x_i}\mid i\in I}$ depends obviously on the Hilbert space dimension of $\mathcal{H}$.
\end{example}

The previous example shows that there does not exist any relationship between the strong Hausdorff separation axiom and the non-equivalence of irreducible representations of $B(\mathcal{H})$ (cf.\ \cite[Thm.~10.2.6 and Prop.~10.4.15]{Kadison2}).




\end{document}